\documentclass[11pt]{article}

\usepackage{amsmath,amscd,amssymb,amsmath,amsthm}
\usepackage{geometry}
\usepackage{enumitem}
\usepackage[utf8]{inputenc} 
\usepackage{authblk}

\title{About the projective Finsler metrizability:
  First steps in the non-isotropic case}

\date{}

\author{T.~Milkovszki and Z.~Muzsnay} 

\allowdisplaybreaks

\theoremstyle{plain} 
\newtheorem{theorem}{Theorem}[section]
\newtheorem{proposition}[theorem]{Proposition}

\newtheorem{corollary}[theorem]{Corollary}
\newtheorem{lemma}[theorem]{Lemma}
\theoremstyle{definition}
\newtheorem{definition}[theorem]{Definition}

\newtheorem{remark}[theorem]{Remark}

\theoremstyle{remark}

\newcommand{\R}{\mbox{$\mathbb R$}} 
\newcommand{\ts}{\textsuperscript}

\def\TM{\mathcal{T}M}
\def\X#1{\mathfrak{X}(#1)}
\def\u#1{\underline{#1}}
\def\Ker{\mathrm{Ker \,}}
\def\Im{\mathrm{Im \,}}

\def\c{{^{\scriptsize c}}}
\def\E{\mathcal{E}}
\def\rank{\mathrm{rank \,}}
\def\nul{\mathrm{nul \,}}
\def\RTM{\mathbb R_{TM}}
\def\L{\mathcal L}
\def\o{\!\otimes \!}
\def\P{\mathcal P}
\def\PP{\mathcal{\widetilde{P}}}
\def\csum{\mathrm{\scriptstyle{cyc}}\!\!\sum}
\def\t{\!\times\!}
\def\+{\!+\!}
\def\={\!=\!}
\def\<{\!<\!}
\def\>{\!>\!}
\def\L{\mathcal L}
\def\C#1{\mathcal C_{#1}}
\def\Cr#1{\mathcal C^r_{#1}}

\def\P{\mathcal P}

\newenvironment{packed_enumerate}{
  \begin{enumerate}[topsep=3pt, partopsep=3pt,leftmargin=17pt]
    \setlength{\itemsep}{3pt}
    \setlength{\parskip}{3pt}
    \setlength{\parsep}{0pt}
  }{\end{enumerate}}

\begin{document}

\maketitle

\begin{abstract}
  We consider the projective Finsler metrizability problem: under what
  conditions the solutions of a given system of second-order ordinary
  differential equations (SODE) coincide with the geodesics of a Finsler
  metric, as oriented curves.  SODEs with isotropic curvature have already
  been thoroughly studied in the literature and have proved to be projective
  Finsler metrizable.  In this paper, we investigate the non-isotropic case
  and obtain new results by examining the integrability of the Rapcs\'ak
  system extended with curvature conditions.  We consider the $n$-dimensional
  generic case, where the eigenvalues of the Jacobi tensor are pairwise
  different and compute the first and the higher order compatibility
  conditions of the system. We also consider the three-dimensional case, where
  we find a class of non-isotropic sprays for which the PDE system is
  integrable and, consequently, the corresponding SODEs are projective
  metrizable.
\end{abstract}

\bigskip

\begin{description}
\item [2000 Mathematics Subject Classification:] 49N45, 58E30,
  53C60, 53C22.
\item[Key words and phrases:] Euler-Lagrange equation, metrizability,
  projective metrizability,
  \\
  geodesics, spray, formal integrability.
\end{description}

\smallskip

\section{Introduction}
\label{sec:introduction}

The projective metrizability problem can be formulated as follows: under what
conditions the solutions of a given system of second-order ordinary
differential equations coincide with the geodesics of some metric space, as
oriented curves. This problem can be seen as a particular case of the inverse
problem of the calculus of variations.  One can consider the Riemannian
\cite{Bryant_Dunajski_Eastwood_2009, Eastwood_Matveev_2008} and the more
general Finslerian version of this problem \cite{Paiva_2005, BuMu_2011,
  BuMu_2012, Crampin_2011, MiMu_2017, Rapcsak_1961, szilasi_vattamany_2002}.
In this paper we examine the latter: starting with a homogeneous system of
second-order ordinary differential equations, which can be identified with a
spray $S$, we seek for a Finsler metric $F$ whose geodesics coincide with the
geodesics of the spray $S$, up to an orientation preserving
reparameterization. For flat sprays this problem was first studied by Hamel
\cite{Hamel_1903} and it is known as the Finslerian version of Hilbert’s
fourth problem \cite{Paiva_2005, Crampin_2011}.  Rapcs\'ak obtained necessary
and sufficient conditions in the general case \cite{Rapcsak_1961}.

There are different approaches to tackle the problem: In \cite{CrMeSa_2012}
the authors use the multiplier method which is, in the context of the inverse
problem of the calculus of variations, probably the most used and studied
approach. In \cite{BuMu_2011} the projective metrizability problem was
reformulated in terms of a first-order partial differential equation and a set
of algebraic conditions on a semi-basic 1-form. Finally, \cite{MiMu_2017}
turns back to the original idea of Rapcs\'ak by considering the so called
Rapcs\'ak system, which is the PDE system composed by the Euler-Lagrange
partial differential equations and the homogeneity condition on the unknown
Finsler function $F$.  The different approaches are equivalent and can lead to
effective results. For example sprays with isotropic curvature was
investigated in all three different approaches and it was proved that this
class of spray is projective metrizable. Unfortunately, beyond the isotropic
case, there are practically no results about this problem in the literature.

In this paper we make a step in the direction to find new results in the
non-isotropic case by investigating the integrability of the Rapcs\'ak
system. Here the difficulties come from the fact that the PDE system is
largely overdetermined: on an $n$-dimensional manifold there are $n+1$
equations on the unknown Finsler function, therefore many integrability
conditions arise. That is why in the generic case there is no solution to the
problem.  In \cite{MiMu_2017} the first compatibility conditions of the
Rapcs\'ak system were already determined: they can be expressed in terms of
equations containing the associated nonlinear connection and its curvature
tensor.  When the spray is isotropic, these conditions are satisfied.
However, when the spray is non-isotropic, the integrability conditions are not
satisfied and further compatibility conditions appear. From that point, the
analysis becomes quite difficult, because of two reasons. First: the curvature
tensor has no canonical normal form, therefore each class of sprays having
different curvature form must be considered separately. Second: as it has been
shown in \cite{MiMu_2017}, the system containing the curvature condition may
be not $2$-acyclic (the Cartan's test fails), that is higher order
compatibility conditions can arise.

The paper is organized as follows.  In Section \ref{sec:2} we give a brief
introduction to the canonical structures on the tangent bundle of a manifold
and the main structures needed to discuss the geometry of a spray: connection,
Jacobi endomorphism, curvature. We also recall the basic tools of
Cartan-K\"ahler theory.  In Section \ref{sec:3} the extended Rapcs\'ak system
with curvature condition is considered in the $n$-dimensional generic case,
when the eigenvalues of the Jacobi curvature tensor $\Phi$ are pairwise
different: in Subsection \ref{sec:3_1} we compute its first compatibility
conditions and in Subsection \ref{sec:3_2}, using the prolonged system, we
find the higher order compatibility conditions. In Section \ref{sec:4} we
consider the $3$-dimensional case: We prove that the symbol of the prolonged
system is 2-acyclic and discuss in detail the reducible case.  Finally, we
identify a class of non-isotropic sprays for which the second and the third
order conditions are identically satisfied and the symbol of the prolonged
system is 2-acyclic. Therefore we obtain the formal integrability of the
system and, in the analytic case, the projective metrizability of this class
of sprays.  At the end of this section we give an example of non-isotropic
SODEs where the results of Section \ref{sec:4} can be applied to prove the
projective metrizability of the systems.

\section{Preliminaries}
\label{sec:2}

In this paper $M$ denotes an $n$-dimensional smooth manifold, $C^\infty(M)$ is
the ring of the smooth functions on $M$ and $\mathfrak X(M)$ is the
$C^\infty(M)$-module of vector fields on $M$. The set of the skew-symmetric,
symmetric and vector valued $k$-forms are $\Lambda^k(M)$, $S^k(M)$ and
$\Psi^k(M)$, respectively.  Furthermore, $\Lambda^k_v(TM)$ stands for the set
of the semi-basic $k$-forms.  The tangent bundle $(TM,\pi,M)$ and the slashed
tangent bundle
\begin{math}
  (\mathcal{T}M:=TM\setminus \left\{0\right\},\pi,M)
\end{math}
of $M$ will simply be denoted by $TM$ and $\mathcal{T}M$.  The tangent bundle
of $TM$ will be denoted by $TTM$ or $T$.  $VTM=\Ker (\pi_*:TTM\rightarrow TM)$
is the vertical sub-bundle of $T$.  

We denote the coordinates on $M$ by $x\!=\!(x^i)$ and the induced coordinates
on $TM$ by $(x,y)\!=\!(x^i,y^i)$.  The local expressions of the Liouville
vector field $C \in \mathfrak{X}(TM)$ corresponding to the infinitesimal
dilatation in the fibres, and the vertical endomorphism $J\in \Psi^{1}(TM)$
are
\begin{align*}
  C = y^i \frac {\partial}{\partial y^i }, \qquad 
  J = dx^i \otimes \frac {\partial}{\partial y^i }.
\end{align*}
Using Euler's theorem on homogeneous functions $f\in C^{\infty}(\TM)$ is
positive homogeneous of degree $k$ if $\mathcal L_Cf=kf$.

A \textit{spray} on $M$ is a vector field $S\in\mathfrak{X}(\mathcal{T}M)$
satisfying the conditions $JS = C$ and $[C,S]=S$. The coordinate expression of
a spray $S$ takes the form
\begin{equation*}
  \label{eq:S}
  S = y ^i \frac {\partial}{\partial x^i } +f^i (x,y)
  \frac {\partial}{\partial y^i },
\end{equation*}
where the spray coefficients $f^{i}=f^{i}(x,y)$ are 2-homogeneous functions.
The \textit{geodesics of a spray} $S$ are curves $\gamma : I \to M$ such that
$S \circ \dot \gamma = \ddot \gamma$. Locally, they are the solutions of the
second order ordinary differential equations
\begin{math}
  \ddot{x}^i=f^i(x, \dot{x}),
\end{math}
$i=1,\dots ,n$. 

A \emph{Finsler function} on a manifold $M$ is a continuous function
$F\colon TM \to \R$, smooth and positive away from the zero section,
homogeneous of degree 1 and the matrix composed by the coefficients
\begin{math}
  g_{ij}=\frac{1}{2}\frac{\partial^{2} F^{2}}{\partial y^{i}\partial y^{j}}
\end{math}
of the metric tensor $g=g_{ij}dx^i\otimes dx^j$ is positive definite on $\TM$.
Consequently, the Hessian of $F$ is positive quasi-definite in the sense that
\begin{math}
  \frac{\partial^2F}{\partial y^i \partial y^j}v^iv^j \geq 0
\end{math}
with equality only if $v$ is a scalar multiple of $y$ (see
\cite{CrMeSa_2012}).  The pair $(M,F)$ is called Finsler manifold.  To any
Finsler function $F$ there exits a unique canonical spray
$S_{\scriptscriptstyle F}$, such that the geodesics of $F$ are the geodesic of
$S_F$. The canonical spray is characterized by the equation
\begin{math}
  i_{S_{\scriptscriptstyle F}} dd_JF^{2}=-dF^{2}.
\end{math}
A spray $S$ is called \textit{Finsler metrizable} if there exist a Finsler
metric $F$ whose canonical spray is $S_{\scriptscriptstyle F}=S$. A spray $S$
is called \textit{projective Finsler metrizable} if there exist a Finsler
metric $F$ whose canonical spray is projective equivalent to $S$ i.e.
$S_{\scriptscriptstyle F} \sim S$.  In that case the geodesics of the two
sprays coincide up to an orientation preserving reparametrization.  According
to Rapcs\'ak's result \cite{Rapcsak_1961}, the spray $S$ is projective Finsler
metrizable if and only if there exists a Finsler function $F$ such that
$i_{S}\Omega=0$, where $\Omega:=dd_JF$ ($\Omega\!\in\!  \Lambda^2(TM)$).
Consequently, the spray $S$ is projective Finsler metrizable if and only if
the \textit{Rapcs\'ak system}
\begin{equation}
  \label{eq:rap_1}
  \big\{ \mathcal{L}_CF-F=0, \quad  i_{S}\Omega=0\big\}
\end{equation}
admits a positive quasi-definite solution $F$.  The first compatibility
conditions of the partial differential system \eqref{eq:rap_1} were determined
in \cite{MiMu_2017} in terms of geometric objects associated to the spray.  To
present them, let us consider the following notions.  The \textit{connection}
associated to $S$ is the vector valued one form $\Gamma: = [J,S]$.  One has
$\Gamma^{2}=\mathrm{Id}$ and the eigenspaces of $\Gamma$ corresponding to the
eigenvalue $+1$ and $-1$ are the vertical and the horizontal subspaces. The
horizontal bundle is denoted by $HTM$. The horizontal and vertical projectors
associated to the connection are
\begin{math}
  h : = \frac{1}{2}(\textrm{I} + \Gamma )
\end{math}
and
\begin{math}
  v : = \frac{1}{2}(\textrm{I} - \Gamma ).
\end{math}
We have $hS=S$, that is the spray is a horizontal vector field.  The
\textit{curvature} $R \in \Psi^{2}(TM)$ of the connection $\Gamma$ is the
Nijenhuis torsion of the horizontal projection
\begin{math}
  R = \frac{1}{2}[h,h].
\end{math}
The \textit{Jacobi endomorphism} $\Phi$ can be derived from the curvature by
$\Phi:=i_{S}R$. They are also related by the equation
\begin{math}
  \frac{1}{3}[J,\Phi]=R.
\end{math}
In \cite{MiMu_2017} it was proved that the spray $S$ is projective metrizable
if and only if there exists a regular function $F$ on $\TM$ such that it is a solution to
the \emph{extended Rapcs\'ak system}:
\begin{equation}
  \label{eq:rap_2}
  \big\{  \mathcal L_CF-F = 0, \quad i_\Gamma \Omega = 0  \big\}.
\end{equation}
The compatibility conditions of the system \eqref{eq:rap_2} can be expressed a
coordinate free way in terms of the curvature tensor $R$ of $\Gamma$ by the
equation $i_R \Omega=0$. In the case when $\dim M=2$ or $S$ is flat or more of
isotropic curvature that is the flag curvature does not depend on the
direction (cf.~\cite{bucataru_muzsnay_2014, Chern_Shen_2006}), then this
compatibility condition is satisfied and the system \eqref{eq:rap_2} is
integrable. One obtains that in the analytic case the spray $S$ is locally
projective metrizable \cite{MiMu_2017}.

In this article we are investigating the projective metrizability when the
curvature of the spray is non-isotropic. In that case the compatibility
condition of \eqref{eq:rap_2} is not satisfied. This is why one has to
consider an enlarged system by adding to \eqref{eq:rap_2} its compatibility
condition. We remark that instead of the curvature tensor, one can express the
compatibility condition in terms of the Jacobi tensor (see \cite{CrMeSa_2012})
and consider the enlarged the system:
\begin{equation}
  \label{eq:rapcsak_3}
  \mathcal L_CF-F = 0, \qquad i_\Gamma \Omega = 0, \qquad i_\Phi \Omega = 0.  
\end{equation}
Depending on the algebraic form of $\Phi$ there are many cases and sub cases
to consider.  Moreover, there is a new phenomenon appearing: the partial
differential system \eqref{eq:rapcsak_3} is not 2-acyclic \cite[Section
5]{MiMu_2017}, which is the indication of the existence of higher order
integrability condition.  All these informations suggest that the
non-isotropic case may be extremely complex.  We consider the generic case,
when $\Phi$ is diagonalizable with distinct eigenvalues in the following
sense: A function $\lambda \in C^{\infty}(TM)$ is called an eigenfunction and
$X\in \X {TM}$ is an eigenvector field of $\Phi$ if $X$ is a horizontal and
$\Phi_u X_u=\lambda_u JX_u$ for any $u\in \TM$.  It is easy to verify that the
spray $S$ is an eigenvector field of $\Phi$ and the corresponding Jacobi
eigenfunction is $\lambda=0$.  

To investigate the integrability of the system \eqref{eq:rapcsak_3} we use the
Spencer-Goldschmidt's integrability theory. Here we just set the notation. For
more details, we refer to \cite{BCGGG_1991, GrMu_2000} and \cite{MiMu_2017} in
the context of projective metrizability. Let $E=(E,\pi,M)$ be a fibred bundle
over the manifold $M$.  $J_{k}(E)$ denotes the bundle of $k$-jets of sections
of $E$.  Then $J_{k+1}(E)$ become a fibred bundle over $J_k(E)$ with the
projection $\pi_k:J_{k+1}(E)\rightarrow J_k(E)$.  Let $E$ and $\widetilde E$
be vector bundles over the manifold $M$ and
$P:Sec(E)\rightarrow Sec(\widetilde E)$ be a $k^\mathrm{th}$ order
differential operator. Then $P$ can be identified with the map
$p_{k}(P)\colon J_{k}E\rightarrow \widetilde E$ and a natural way the
$l^{\mathrm{th}}$ prolongation can be introduced as
\begin{math}
  p_{k+l}(P)\colon J_{k+l}E\rightarrow J_{l}\widetilde E.
\end{math}
The elements of
\begin{math}
  Sol_{k+l} := \mathrm{Ker}\, p_{k+l}(P)
\end{math}
is called the $l\ts{th}$ order formal solutions.  $P$ is \textit{formally
  integrable} if $Sol_{l}$ is a vector bundle over $M$ for all $l \geq k$, and
the map $\overline{\pi}_{l}: Sol_{l+1}\rightarrow Sol_{l}$ is onto $\forall$
$l \geq k$. In that case any $k^{\mathrm{th}}$ order solution can be lift into
an infinite order formal solution.  The highest order terms of $P$ said to be
the symbol of $P$. It can be interpreted as a map
$\sigma_k\colon S^{k}T^* \! M\otimes E \rightarrow \widetilde E$.  The $l\ts{th}$ order
prolongation of the symbol is denoted as
\begin{math}
  \sigma_{k+l}\colon S^{k+l}T^* \!M\otimes E\rightarrow S^lT^*M\otimes \widetilde E.
\end{math}

The existence or the non-existence of higher order compatibility condition can
be calculated classically with the Cartan's test but more information can be
obtained about the higher order compatibility conditions from the Spencer
cohomology groups.  Classical version of the Cartan-K\"ahler integrability
theorem uses the Cartan's test while its generalization, proved by
Goldschmidt, uses the Spencer cohomology groups.  Let us introduce the two
concepts.  A basis $(e_i)^{n}_{i=1}$ of $T_xM$ is \textit{quasi-regular} if
\begin{math}
  \dim g_{k+1,x}\!=\!\dim g_{k,x}\!+\!\sum^{n}_{j=1}\dim
  (g_{k,x})_{e_1,...,e_j},
\end{math}
where one notes $g_{k+l}:=\mathrm{Ker} \, \sigma_{k+l}$ and
\begin{math}
  (g_{k,x})_{e_1,...,e_j}:=\left\{A \in g_{k,x} | \, i_{e_1} A =0, \dots
    i_{e_j}A=0\right\},
\end{math}
\ $(1\leq j\leq n)$.  The symbol $\sigma_k$ is \textit{involutive} if there
exists a quasi-regular basis of $T_xM$ at any $x \in M$ (Cartan's test).
Moreover, the symbol of $P$ is called \textit{2-acyclic} if for all $m\geq k$
the $(m,2)$ Spencer cohomology groups
\begin{math}
  H^{m,2}=\mathrm{Ker}\, \delta^{m}_{2}/\mathrm{Im}\,\delta^{m}_{1}
\end{math}
vanish where
\begin{alignat*}{1}
  g_{m+1} \!\otimes \! T^*\!M \stackrel{\delta^{m}_{1}}{\longrightarrow}
  g_{m} \!\otimes\! \Lambda^{2}T^*\!M, \qquad
  g_{m} \!\otimes\! \Lambda^2T^*\!M \stackrel{\delta^{m}_{2}}{\longrightarrow}
  g_{m-1}\!\otimes\!  \Lambda^{3}T^*\!M
\end{alignat*}
are the natural skew-symmetrizations. We remark that if $\sigma_k(P)$ is
\textit{involutive} then all Spencer cohomology groups are zero.

\begin{theorem}[Cartan-K{\"a}hler/Goldschmidt]
  \label{Cartan_Kahler}
  Let $P\colon J_kE\rightarrow \widetilde E$ be a $k\ts{th}$ order regular
  linear partial differential operator. If
  \begin{math}
    \overline{\pi}_{k}:Sol_{k+1}\rightarrow Sol_{k}
  \end{math}
  is surjective and $\sigma_k$ is involutive/2-acyclic, then $P$ is formally
  integrable.
\end{theorem}

\begin{remark}[Computation of the first compatibility condition]
  \label{sec:remark_surjectivity} \
  \\
  In the practice the surjectivity of $\overline{\pi}_{k}$, or in other words
  the first compatibility condition, can be computed as follows: Using the
  snake lemma of homological algebra one can show that there exits a morphism
  $\varphi$ such that the sequence
  \begin{equation}
    \label{eq:varphi}
    Sol_{k+1}\xrightarrow{ \ \overline{\pi}_k \ } Sol_k\xrightarrow{ \ \varphi
      \ } \mathrm{Coker}(\sigma_{k+1})
  \end{equation}
  is exact.  From the exactness of \eqref{eq:varphi} we get that
  $\overline{\pi}_k$ is onto if and only if $\varphi=0$.  The partial
  differential equation $\varphi=0$ is called the \emph{first compatibility
    condition} of $P$.  To compute $\varphi$ we note that if
  $\tau:T^{*}\otimes \widetilde E\rightarrow K$ is a morphism such that
  $\mathrm{Ker}\, \tau=\mathrm{Im} \, \sigma_{k+1}$ then $\mathrm{Im}\, \tau$
  is isomorphic to $\mathrm{Coker}(\sigma_{k+1})$ and
  \begin{equation}
    \label{eq:int_cond}
	\varphi=\tau \circ \nabla P|_{Sol_k} ,
  \end{equation}
  where $\nabla$ is an arbitrary linear connection on $F$ (see
  \cite[p.~28]{GrMu_2000}).
\end{remark}

\section{Extended Rapcs\'ak system with curvature condition}
\label{sec:3}

Let $S$ be a spray on $M$ and 
\begin{math}
  \mathcal P \colon C^\infty(TM) \longrightarrow C^\infty(TM)\times
  \Lambda_v^2(TM)\times \Lambda_v^2(TM)
\end{math}
the differential operator corresponding to the second order linear partial
differential system \eqref{eq:rapcsak_3}:
\begin{equation}
  \label{eq:P}
  \mathcal{P}:=(P_{\Gamma}, P_C, P_{\Phi})
\end{equation}
where
\begin{displaymath}
  P_CF:= \mathcal{L}_CF-F, \qquad
  P_\Gamma F:= i_\Gamma \Omega, \qquad
  P_\Phi F:= i_\Phi \Omega,
\end{displaymath}
with $\Omega:=dd_JF$. We suppose that the Jacobi endomorphism $\Phi$ has $n$
distinct eigenvalues. In this chapter we compute the first integrability
conditions of order 2, and the the higher order compatibility condition of
order 3 of system \eqref{eq:P}.  It turns out that in some cases, even though
the Cartan's test fails, these compatibility conditions give the complete set
of obstructions to the integrability and in very specific situations the
system becomes integrable (see Chapter \ref{sec:4}).

\subsection{First compatibility conditions}
\label{sec:3_1}

First we remark that $hS=S$ that is the spray is horizontal with respect to
the connection associated. Moreover, from the definition of the Jacobi
endomorphism we get
\begin{equation}
  \label{eq:Phi_eigenvalue}
 \Phi(S) = i_SR(S)=R(S,S)=0,
\end{equation}
that is $S$ is an eigenvector of $\Phi$ and the corresponding eigenvalue is
$\lambda=0$.  Let $\lambda_1, \dots ,\lambda_n$ be the $n$ distinct
eigenfunctions of $\Phi$ and $h_1$, \dots, $h_n$ the corresponding eigenvector
fields, where $\lambda_n\!=\!0$ and $h_n\!=\!S$. For any $x\in \TM$ we
consider the basis
\begin{equation}
  \label{eq:basis_1}
  \mathcal B:=\left\{h_1,\dots ,h_n,v_1,\dots,v_n\right\} \quad \subset T_xTM,
\end{equation}
where $Jh_i\!=\!v_i$, $i\!=\!1,\dots,n$.  We have the following
\begin{proposition}  
  \label{prop:2_order_integr_cond}
  A 2\ts{nd} order solution $s=j_{2}(F)_x$ of $\mathcal{P}$ at
  $x\in \mathcal{T}M$ can be lifted into a 3\ts{rd} order solution, if and
  only if
  \begin{alignat}{1}
    \label{iphiphi}
    i_{[\Phi,\Phi]}\Omega_x=0,
    \\
    \label{omegavihjhk}
	\sum_{ijk}^{cycl}(\Omega_x([v_i,h_j],h_k))_x=0.
  \end{alignat} 
\end{proposition}

\noindent
To compute the first compatibility conditions we use the method described in
Remark \ref{sec:remark_surjectivity}.  The symbol of the system $\P$ is
composed by the symbol of the operator $P_C$, $P_\Gamma$ and $P_\Phi$.  The
symbol of the first order operator $P_C$ is
\begin{alignat*}{2}
  \sigma_1(P_C)\colon & T^* \to \R, \qquad \quad &&   \sigma_1( P_C)A_1 =
  A_1(C).
  \intertext{The prolongation of the symbol of $P_C$ and the symbol of
    the second order $P_\Gamma$ and $P_{\Phi}$ are}
  \sigma_2(P_C)\colon & S^2T^* \to T^*, \qquad \quad &&
  \big(\sigma_2(P_C)A_2\big)(X) = A_2(X, C),
  \\
  \sigma_{2}(P_{\Gamma})\colon & S^2T^* \to \Lambda^{2}T^{*}_{v},
  \quad && \big(\sigma_2(P_\Gamma)A_2\big)(X,Y) =
  2\big(A_2(hX,JY)\!-\!A_2(hY,JX)\big),
  \\
  \sigma_{2}(P_{\Phi})\colon & S^2T^* \to \Lambda^{2}T^{*}_{v},
  \quad && \big(\sigma_2(P_\Phi)A_2\big)(X,Y) = A_2(\Phi X,JY)\!-\!A_2(\Phi
  Y,JX),
  \intertext{and their prolongations at third order level are}
  \sigma_3(P_C)\colon & S^3T^* \to S^{2}T^*, \quad &&
  \big(\sigma_3(P_C)A_3\big)(X,Y) = A_3(X,Y, C),
  \\
  \sigma_{3}(P_{\Gamma})\colon & S^3T^* \to T^{*}\!\!\otimes\!
  \Lambda^{2}T^{*}_{v}, \ && \big(\sigma_3(P_\Gamma)A_3\big)(X,Y,Z) \!=\!
  2\big(A_3(X,hY,JZ)\!-\!A_3(X,hZ,JY)\big),
  \\
  \sigma_{3}(P_{\Phi})\colon & S^3T^* \to T^* \! \otimes\!
  \Lambda^{2}T^{*}_{v}, \ && \big(\sigma_3(P_\Phi)A_3\big)(X,Y,Z) \!=\!
  A_3(X,\Phi Y,JZ)\!-\!A_3(X,\Phi Z,JY),
\end{alignat*}
where $X,Y,Z\! \in\! T$, $A_i \!\in\! S^{i}T^{*}$, $i=1,2,3$. Therefore  
\begin{alignat*}{1}
  \sigma_2(\mathcal{P})&=\big(\sigma_2(P_C), \sigma_2(P_\Gamma),
  \sigma_2(P_\Phi)\big) \colon S^2T^*\longrightarrow T^{*} \times
  \Lambda^{2}T^{*}_{v} \times \Lambda^{2}T^{*}_{v},
  \\
  \sigma_3(\mathcal{P})&=\big(\sigma_3(P_\Gamma), \sigma_3(P_C),
  \sigma_3(P_\Phi)\big) \colon S^3T^*\longrightarrow S^{2} T^{*} \times (T^{*}
  \o\Lambda^{2}T^{*}_{v})\times (T^{*}\o \Lambda^{2}T^{*}_{v}).
\end{alignat*}
According to Remark \ref{sec:remark_surjectivity}, to compute the
compatibility condition we should construct a map $\tau$ such that
$\Ker \tau\!=\!\Im \sigma_3(\mathcal{P})$.  Let us consider the map
\begin{equation}
  \label{eq:tau}
  \tau :=  (\tau_1, \tau_2,\, \tau_3,  \tau_4, 
  \tau_5,  \tau_6,  \tau_7, \tau_{ijk}),
\end{equation}
as follows: if $B=(B_C, B_\Gamma, B_\Phi)$ is an element of
\begin{math}
  S^{2} T^{*} \! \times (T^{*} \o\Lambda^{2}T^{*}_{v}) \times (T^{*}\o
  \Lambda^{2}T^{*}_{v})
\end{math}
then
\begin{alignat*}{2}
  &\tau_1(B)(X,Y,Z)=B_{\Gamma}(hX, Y, Z) + B_{\Gamma}(hY, Z, X)
  +B_{\Gamma}(hZ, X, Y),
  \\
  &\tau_2 (B)(X,Y,Z)=B_{\Gamma}(JX, Y, Z)
  +B_{\Gamma}(JY, Z, X) +B_{\Gamma}(JZ, X, Y),
  \\
  &\tau_3(B)(X,Y)=\tfrac{1}{2}B_\Gamma (C,X,Y)-B_C(hX,JY)+B_C(hY,JX),
  \\
  &\tau_4(B)(X,Y,Z)=B_{\Phi}(\Phi X, Y, Z) + B_{\Phi}(\Phi Y, Z,
  X) +B_{\Phi}(\Phi Z, X, Y),
  \\
  &\tau_5 (B)(X,Y,Z)=B_{\Phi}(JX, Y, Z) +B_{\Phi}(JY, Z, X) +B_{\Phi}(JZ, X,
  Y),
  \\
  &\tau_6 (B)(Y,Z)=B_{\Phi}(C, Y,Z)-B_C(\Phi Y, JZ)+B_C(\Phi Z,JY),
  \\
  &\tau_7 (B)(X,Y)=B_{\Phi}(X,Y,S)-B_C(X,\Phi Y),
  \\
  & \tau_{ijk}(B)= \tfrac{1}{2}B_\Gamma(v_i,h_j,h_k)+
  \tfrac{1}{\lambda_k\!-\!\lambda_i}B_\Phi(h_j, h_i,h_k)+ \tfrac{1}{\lambda_i
    \!-\! \lambda_j}B_\Phi (h_k, h_i,h_j),
\end{alignat*}
where $i,j,k$ $(1\leq i,j,k \leq n)$ are pairwise different indices.

\begin{lemma}
  With the above notation we have $\Ker \tau\!=\!\Im \sigma_3(\mathcal{P})$.
\end{lemma}

\begin{proof}
  It is easy to check that $\sigma_3(\mathcal{P}) \circ \tau =0$ therefore 
  \begin{math}
    \Im \sigma_3(\mathcal{P})\subset \mathrm{Ker} \ \tau.
  \end{math}
  Let us compute the dimension of the two spaces.  First we determine the
  dimension of
  \begin{math}
    \Ker \sigma_3 (\mathcal{P})
  \end{math}
  by computing the free tensor components of its elements.  We denote in the
  sequel the components of a symmetric tensor $A\in S^k T^*$ with respect to
  \eqref{eq:basis_1} as
  \begin{equation}
    \label{eq:B_notation}
    A_{i_1\dots i_j \, \u{i_{j+1}}\dots \u{i_k}}:=
    A(h_{i_1}, \dots, h_{i_j},v_{i_{j+1}},\dots, v_{i_k}),
  \end{equation}
  that is the index with respect to vertical vector will be underlined.  Let
  us compute $\dim (g_3(\mathcal{P}))$.  For a symmetric tensor
  $A\in S^{3}T^{*}$ we have
  \begin{subequations}
    \begin{align}
      \label{eq:symb_1}
      A \in \Ker \sigma_3(P_C) \quad \Leftrightarrow \quad & A_{ij\u n}= 0, \
      A_{i\u{j n}}= 0, \ A_{\u{i j n}}=0, \quad &&
      \\
      \label{eq:symb_2}
      A \in \Ker \sigma_3(P_\Gamma) \quad \Leftrightarrow \quad & A_{ij \u
        k}=A_{ik \u j}, \ A_{\u i j \u k}=A_{\u i k \u j}, \quad &&
      i,j,k=1,\dots ,n,
      \\
      \label{eq:symb_3}
      A\in \Ker \sigma_3(P_{\Phi})\quad \Leftrightarrow \quad & A_{\u{ijk}}=0,
      \ A_{i \u{jk}}=0, \quad && i,j,k=1,\dots ,n, \ j\neq k
    \end{align}
  \end{subequations}
  and we use the notation $\C{n,k}\=\binom{n}{k}$ and
  $\Cr{n,k}\=\binom{n+k-1}{k}$.  We have
  \begin{math}
    \Ker \sigma_3 (\mathcal{P})=\Ker \sigma_3(P_C) \,\cap\,\Ker \sigma_3
    (P_{\Gamma}) \, \cap\, \Ker \sigma_3(P_{\Phi})
  \end{math}
  therefore, the tensor components of $A \in g_3(\mathcal{P})$ must satisfy
  all the equations \eqref{eq:symb_1}--\eqref{eq:symb_3}. For
  $1\leq i,j,k\leq n$ we have
  \begin{packed_enumerate}
  \item [$\cdot$] the $(A_{ijk})$ block is totally symmetric therefore, there
    are $\Cr{n,3}$ free components,
  \item [$\cdot$] the $(A_{ij\u{k}})$ block is totally symmetric, because of
    \eqref{eq:symb_2}. Moreover, from \eqref{eq:symb_1} we have
    $A_{ij\u{n}}=A_{in\u{j}}=A_{ni\u{j}}=0$. Therefore, we have $\Cr{n-1,3}$
    free components.
  \item [$\cdot$] there is only $\Cr{n-1,1}$ free components for each of the
    blocks $(A_{i\u{jk}})$ and $(A_{\u{ijk}})$. These are $A_{i\u{ii}}$ and
    $A_{\u{iii}}$, $i=1,...,n-1$.
  \end{packed_enumerate}
  Hence we obtain that
  \begin{math}
    \dim g_3(\mathcal{P})=\Cr{n,3}+\Cr{n-1,3}+2\,\Cr{n-1,1}
  \end{math}
  and
  \begin{equation}
    \label{eq:rang_sigma_3_P}
	\mathrm{rank}\, \sigma_3(\mathcal{P})  =\dim S^{3}T^{*}
    -\dim g_3(\mathcal{P})=\frac{6n^{3}+9n^{2}-9n+12}{6}.
  \end{equation}
  On the other hand, let us compute
  $\nul (\tau)\! =\! \dim \mathrm{Ker}\, \tau$. If
  $B=(B_C, B_\Gamma, B_\Phi)\in \mathrm{Ker}\, \tau$ then, using the notation
  \eqref{eq:B_notation} for the tensor components we have
  \begin{packed_enumerate}
  \item [$\cdot$] the pivot terms for $\tau_1=0$ and $\tau_2=0$ are
    $B^{\Gamma}_{ijk}$ and $B^{\Gamma}_{\u i j k}$, $i<j<k$ hence the there
    are $2\,\C{n,3}$ independent equations here.
  \item [$\cdot$] The pivots for equation $\tau_3=0$ are
    $B^{\Gamma}_{\u nij}$, $i<j$ giving $\C{n,2}$ independent equations.
  \item [$\cdot$] The pivot terms for $\tau_4=0$ and $\tau_5=0$ are
    $B^{\Phi}_{\u i jk}$, $i\<j\<k\<n$ and $k\<i\<j\<n$ respectively. Hence we
    have here $2\,\C{n-1,3}$ independent equations.
  \item [$\cdot$] The pivot terms for $\tau_6=0$ are $B^{\Phi}_{\u ij}$,
    $i<j$. It gives $\C{n,2}$ independent equations.
  \item [$\cdot$] The pivot terms for $\tau_7=0$ are $B^{\Phi}_{\u i j n}$,
    $i,j\neq n$ and $B^{\Phi}_{ijn}$, $j\neq n$.  It gives
    $\C{n-1,1}\,\C{n-1,1}+\C{n,1}\,\C{n-1,1}$ independent equations.
  \item [$\cdot$] The equations \ $\tau_{ijk}=0$ \ are not all independent,
    because, for any $i,j,k$ we have
    \\
    \begin{math}
      \sum_{cycl}\left(\tfrac{1}{2}B^\Gamma_{\u ijk}+
      \frac{1}{\lambda_k-\lambda_i}B^\Phi_{jik}+
      \frac{1}{\lambda_i-\lambda_j}B^\Phi_{kij} \right) = 0. 
    \end{math}
    Taking into consideration these relations, the pivots terms are
    $B^\Gamma_{\u k ij}$, $i\<j\<k\<n$ and $B^\Gamma_{\u j ki}$, $i\<j\<k$,
    and there are $\C{n-1,3}+\C{n,3}$ independent equations.
  \end{packed_enumerate}
  Subtracting from the dimension of the domain of $\tau$ the number of
  in\-de\-pen\-dent equations we find
  \begin{equation}
    \label{pivots2}
    \begin{aligned}
      \nul (\tau)& = \Cr{2n,2}+ 2\!\cdot \!2n\, \C{n,2}
      -\bigg(2\,\C{n,3}\+\C{n,2}\+2\,\C{n-1,3}\+\C{n,2} +
      \\
      &+\C{n-1,1}\,\C{n-1,1} \+ \C{n,1}\,\C{n-1,1} \+ \C{n-1,3} \+
      \C{n,3}\bigg) =\frac{6n^{3}+9n^{2}-9n+12}{6}.
	\end{aligned}
  \end{equation}
  Comparing \eqref{eq:rang_sigma_3_P} and \eqref{pivots2} we get the statement
  of the lemma.
\end{proof}

\bigskip

\noindent
\emph{Proof of Proposition \ref{prop:2_order_integr_cond}}. The following commutative
diagram shows the maps introduced above:
\begin{displaymath}
  \begin{CD}
    g_3(\mathcal{P})@>i>> S^{3}T^{*} @>\sigma_3(\mathcal{P})>> S^{2}T^{*}
    \! \times \! (T^{*} \o \Lambda^{2}T^{*}_{v})\!\times \!(T^{*}
    \o\Lambda^{2}T^{*}_{v})@>\tau>> K \longrightarrow 0
    \\
    @VVV @VV \epsilon V @VV\epsilon V
    \\
    Sol_3(\mathcal{P}) @> i >> J_3\mathbb{R} @>p_3(\mathcal{P})>>
    J_2(\RTM)\times J_1(\Lambda^{2}T^{*}_{v}\times \Lambda^{2}T^{*}_{v})
    \\
    @VV \overline{\pi}_2V @VV \pi_2 V @VV V
    \\
    Sol_2(\mathcal{P}) @> i >> J_2\mathbb{R} @>p_2(\mathcal{P})>>
    J_1(\RTM)\times \Lambda^{2}T^{*}_{v}\times \Lambda^{2}T^{*}_{v}
  \end{CD}
\end{displaymath}
To compute the first compatibility condition of the system represented by $\P$
we use the method described in Remark \ref{sec:remark_surjectivity}.  Let $F$
be a second order solution of $\mathcal{P}$ at a point $x$, that is
$(\mathcal{P}F)_x=0$. We have
\begin{equation}
  \label{eq:P_3_x}
  (\L_CF\!-\!F)_x\!=\!0, \quad   \big(\nabla(\L_CF\!-\!F)\big)_x\!=\!0, 
  \quad  (i_\Gamma dd_JF)_x\!=\!0,  \quad (i_{\Phi}dd_JF)_x=0.
\end{equation}
Using the map $\tau$ given in \eqref{eq:tau} the integrability condition is
given by formula \eqref{eq:int_cond}:
\begin{enumerate}
\item
  \begin{math}
    \begin{aligned}[t]
      \displaystyle \tau_1(\nabla \P F)_x& =d_{h}(i_{\Gamma}dd_JF)_x
      =(d_hi_{2h-I}dd_JF)_x =(2(d_hi_hdd_JF-d_hdd_JF))_x=
      \\
      &=(2d_h(i_hd-di_h)d_JF)_x=(2d_hd_hd_JF)_x\frac{}{}
      =(d_Rd_JF)_x=(i_R\Omega)_x.
    \end{aligned}
  \end{math}
\item 
  \begin{math}
    \begin{aligned}[t]
      \tau_2 (\nabla & \P F)_x= d_{J}(i_{\Gamma} \Omega)_x \! = \!
      (d_J(i_{2h-I}\Omega))_x\!=\!(2d_Ji_hdd_JF-2d_Jdd_JF)_x
      \\
      &\!=\!(-2d_Ji_hd_JdF \!-\!2i_Jddd_JF\!+\!2di_Jdd_JF)_x
      \!=\!-(2d_J(d_Ji_hdF+d_JdF))_x\!=\!0
    \end{aligned}
    \medskip
  \end{math}
  \\
  $\frac{}{}$ \qquad where we used $[d,d_J]=0$, $[i_h,d_J]=d_{Jh}-i_{[h,J]}$
  and $[J,h]=0$.
\item
  \begin{math}
    \begin{aligned}[t]
      \tau_3(\nabla & \P F )_x(X,Y)
      =\tfrac{1}{2}\nabla i_{\Gamma}\Omega(C,X,Y)-\nabla P_CF(hX,JY)+\nabla
      P_CF(hY,JX)
      \\
      &=\tfrac{1}{2}d_Ci_{\Gamma}\Omega(hX,hY)
      -\tfrac{1}{2}i_{\Gamma}d_Cdd_JF(hX,hY)
      =\tfrac{1}{2}d_{[C,\Gamma]}\Omega(hX,hY) \stackrel{[C,\Gamma]=0}{=}0.
    \end{aligned}
  \end{math}
\item 
  \begin{math}
      \begin{aligned}
        \tau_4(\nabla \P F)_x =(d_{\Phi}(i_{\Phi}\Omega))_x
        =(d_{\Phi}d_{\Phi}d_JF+d_{\Phi}di_{\Phi}d_JF)_x
        =\tfrac{1}{2}d_{[\Phi,\Phi]} d_JF_x =\tfrac{1}{2}i_{[\Phi,\Phi]}
        \Omega_x.
      \end{aligned}
\end{math}

\item
  \begin{math}
    \begin{aligned}
      \tau_5(\nabla \P F)_x=(d_J(i_{\Phi}\Omega))_x =
      (-i_{[J,\Phi]}\Omega+i_{\Phi}d_J\Omega)_x=(-3i_R\Omega)_x=0
    \end{aligned}
  \end{math}
  \vspace{10pt}
  \\
  $\frac{}{}$ \qquad because we have the identity
  $[i_{\Phi},d_J]=d_{J\Phi}-i_{[\Phi,J]}$.
\item Using the notation $F_c:=CF-F$, we have
    \begin{equation*}
      \begin{aligned}
        i_{\Phi}dd_JF_c(X,Y)_x=dd_JF_c(\Phi X,Y)-dd_JF_c(\Phi Y,X)=\Phi
        X(JYF_c)-\Phi Y(JXF_c)
      \end{aligned}
    \end{equation*}
    and 
    \begin{math}
      \begin{aligned}[t]
        \tau_6&(\nabla \P F)_x =(d_C(i_{\Phi}\Omega)-i_{\Phi}dd_JF_c)_x
        =(d_Cd_{\Phi}d_JF-d_{\Phi}d_Jd_CF)_x
        \\
        &=(d_Cd_{\Phi}d_JF-d_{\Phi}d_Cd_JF-d_{\Phi}d_JF)_x
        =(d_{[C,\Phi]}d_JF)_x=(i_{\Phi}\Omega)_x=0.
      \end{aligned}
    \end{math}
  \item
  \begin{math}
    \begin{aligned}[t]
      \tau_7&(\nabla \P F)_x (X,Y) =X_x(i_{\Phi}dd_JF(Y,S))-X_x(\Phi Y(CF-F))
      \\
      &=X_x(\Phi Y d_JF(S)-d_JF([\Phi Y,S]))-X_x(\Phi Y(CF-F))
      \\
      &=-X_x(J[\Phi Y,S]F)+X_x(\Phi YF)=-X_x([J,S](\Phi Y)F)+X_x(\Phi YF)
      \\
      &=-X_x((2h-I) \Phi YF)-X_x(\Phi YF)=0.
    \end{aligned}
  \end{math}
\item Since $i_J\Omega(h_k,h_l)=0$ we have
  $\Omega(v_k,h_l)=\Omega(v_l,h_k)$. Moreover,
  \begin{displaymath}
    i_\Phi\Omega(h_i,h_k) =\Omega(\Phi h_i,h_k)-\Omega(\Phi h_k,h_i)=
    \lambda_i\Omega(v_i,h_k) -\lambda_k\Omega(v_k,h_i) =
    (\lambda_i-\lambda_k)\Omega(v_i,h_k),
  \end{displaymath}
  therefore, using the fact that $\lambda_i \! \neq \! \lambda_j$ and $F$ is a
  $2^{nd}$ order solution at $x$, from the last equation of \eqref{eq:P_3_x} we
  obtain $\Omega_x(v_i,h_k)=0$. Consequently,
  \begin{displaymath}
    \begin{aligned}[t]
      \nabla i_\Phi\Omega(h_j, h_i,h_k)_x =&
      \big(h_j(i_\Phi\Omega(h_i,h_k))\big)_x
      =\big((\lambda_i-\lambda_k)h_j\Omega(v_i,h_k)\big)_x
    \end{aligned}
  \end{displaymath}
  and
  \begin{displaymath}
    \begin{aligned}[t]
      \tau_{ijk}&(\nabla \P F)_x(h_i,h_j,h_k)=
      \\
      &=\tfrac{1}{2}\nabla i_\Gamma\Omega(v_i,h_j,h_k)_x+
      \frac{1}{\lambda_k-\lambda_i}\nabla i_\Phi\Omega(h_j, h_i,h_k)_x+
      \frac{1}{\lambda_i-\lambda_j}\nabla i_\Phi\Omega(h_k, h_i,h_j)_x
      \\
      &=v_i\Omega(h_j,h_k)_x+h_j\Omega(h_k,v_i)_x+h_k\Omega(v_i,h_j)_x.
    \end{aligned}
  \end{displaymath}
  From $d\Omega=0$ we have \
  \begin{math}
    \sum_{ijl}^{cycl} v_i\Omega(h_j,h_k) -\Omega([v_i,h_j],h_k)=0,
  \end{math}
  thus
  \begin{displaymath}
	\tau_{ijk}(\nabla \P F)_x 
    =\Omega_x([v_i,h_j],h_k) +\Omega_x([h_j,h_k],v_i)+\Omega_x([h_k,v_i],h_j).
  \end{displaymath}
\end{enumerate}
These calculations shows that the first compatibility condition
\eqref{eq:int_cond} for $\P$ is
\begin{displaymath}
  \varphi (F)_x=\tau (\nabla \P F)_x = 
  \big(i_R \Omega, \, 0, \, 0, \, \tfrac{1}{2}i_{[\Phi, \Phi]}\Omega,  
    \,     0, \, 0, \,  0, \,  \sum^{cycl}_{ijk} \Omega([v_i,h_j],h_k)
  \big)_x
\end{displaymath}
which  proves the proposition.

\hfill \rule{1ex}{1ex}

\subsection{Higher order compatibility condition}
\label{sec:3_2}

The integrability theorem of Cartan-K{\"a}hler or Spencer-Goldschmidt states
that if the first compatibility equations are satisfied and the symbol of the
linear differential operator is surjective (resp.~2-acyclic), then there is no
more (i.e. higher order) compatibility condition and the PDE operator is
formally integrable. Unfortunately this is not the case with the PDE operator
\eqref{eq:P}.  As it was proved in \cite[Theorem 5.1]{MiMu_2017}, the symbol
of $\P$ is not $2$-acyclic and therefore it is not involutive either. Since
the first non trivial Spencer cohomology group is $H^{2,2}$, we can deduce
that at the level of first prolongation of the PDE operator $\P$ there is an
extra (third order) compatibility condition appearing. In this subsection we
compute this extra condition and prove the following

\begin{proposition}
  \label{prop:3_order_integr_cond}
  If a 3\ts{rd} order solution $F$ of $\mathcal{P}$ at $x\in \mathcal{T}M$ can
  be lifted into a 4\ts{th} order solution, if and only if for any
  $X,Y\in H_xTM$ we have
  \begin{alignat}{1}
    \notag
    [&hX,\Phi X] \Omega(JY,Y)- [hY,\Phi Y]\Omega(X,JX) = \vphantom{\sum_{cycl}}
    \\
    \label{third_order_cond}
    & \ + \Phi X\Big(\sum_{cycl}\Omega([JY,hX],hY)\Big)
    - hX\Big(\sum_{cycl} \Omega([JY,\Phi X],hY) -\Omega([JY,\Phi Y],X)\Big)
    \\
    \notag & \  -\Phi Y\Big(\sum_{cycl}\Omega([JX,hX],hY)\Big)
    +hY\Big(\sum_{cycl}\Omega([JX,\Phi X],Y) -\Omega([JX,\Phi Y],X)\Big)
  \end{alignat}
\end{proposition}

\begin{proof}
  If the Spencer cohomology groups $H^{2,2}$ would be trivial, then the
  compatibility condition of the first prolongated system would be exactly the
  prolongation of the first compatibility condition of $\P$. However, from
  Theorem 5.1 in \cite{MiMu_2017} we know that
  \begin{math} 
    \dim H^{2,2} =\frac{1}{2}(n\!-\!1)(n\!-\!2)=\C{n,2}.
  \end{math}
  That indicates that there is $\C{n,2}$ extra compatibility condition for the
  first prolongation of $\P$. To compute these conditions we will complete the
  prolongation of the map $\tau$ defined in \eqref{eq:tau} and apply the
  method described in Remark \ref{sec:remark_surjectivity}.

  \smallskip

  \noindent
  Let us consider the sequence
  \begin{equation}
    \label{egzaktseq}
    S^{4}T^{*}\xrightarrow{ \ \sigma_4(\P) \ } 
    (S^{3}T^{*}) \! \times \! (S^{2}T^{*}\o \Lambda^{2}T^{*})\!\times\!   
    (S^{2}T^{*}\o \Lambda^{2}T^{*})\xrightarrow{ \ \  \tau^1  \ \ }
    \  K^1 \ \longrightarrow \  0
  \end{equation}
  where 
  \begin{math}
    \sigma_4(\P) \!= \!id \otimes \sigma_3(\P)
  \end{math}
  is the prolongation of the symbol of $\P$, and the components of
  $\tau^1=(id \o \tau, \tau_h)$ are the prolongation $id \otimes \tau$ of
  \eqref{eq:tau} and $\tau_h$ with
  \begin{align}
    \label{highero}
    \begin{split}
      \tau_{h} (B)(X,Y):=&\frac{1}{2}\left(B_\Gamma (\Phi X, JY, X,Y)-B_\Gamma
        (\Phi Y, JX, X,Y)\right)
      \\
      &+ B_\Phi (hY, JX, X,Y) - B_\Phi (hX, JY, X,Y),
    \end{split}
  \end{align}
  where
  \begin{math}
    B:=(B_C, B_{\Gamma}, B_{\Phi})\in S^{3}T^* \t (S^{2}T^{*}\o
    \Lambda^{2}T^{*}_{v}) \t (S^{2}T^{*}\o \Lambda^{2}T^{*}_{v}),
  \end{math}
  $X,Y\in T$.  It is not difficult to show that the equation $\tau_h=0$ gives
  $\C{n,2}$ extra independent equations with respect to the equations of the
  system $\tau^1=0$, thus the sequence \eqref{highero} is exact.  \bigskip

  \noindent
  Let us compute the compatibility conditions appearing from the new
  obstruction map $\tau_h$. If $F$ is a $3\ts{rd}$ order solution of
  $\mathcal{P}$ at $x\in \mathcal{T}M$, then
  \begin{alignat*}{1}
    & \tau_{h}(\nabla \nabla \mathcal{P}F)_x(X,Y) =
    \\
    & =\frac{1}{2}\left(\!\nabla^{2} i_\Gamma\Omega(\Phi X, \! JY,\!
      X,\!Y)\!-\!\nabla^{2} i_\Gamma \Omega(\Phi Y, \!JX, \!X,\!Y)\right)
    \! + \! \nabla^{2} i_\Phi \Omega(hY,\! JX,\! X,\!Y) \!-\! \nabla^{2}
    i_\Phi \Omega (hX, \!JY, \!X,\!Y)
    \\
    & = \Phi X(JY i_\Gamma \Omega(X,Y)) \!-\!\Phi Y(JXi_\Gamma\Omega(X,Y))
    \frac{}{}
    \!+\! hY(JX i_\Phi \Omega(X,Y)) \!- \!hX(JY i_\Phi \Omega(X,Y)).
  \end{alignat*}
  In the above formula the $2\ts{nd}$ derivatives of $\Omega$ (and therefore
  4\ts{th} derivatives of $F$) appears. However, using the fact that
  $\Omega=dd_JF$ vanishes identically on the vertical sub-space,
  $\Omega(JX,hY)=\Omega(JY,hX)$ and $\Omega(Y,JX)=\Omega(JY,X)=0$, the
  $2\ts{nd}$ derivatives of $\Omega$ can be expressed by $1\ts{st}$
  derivatives of $\Omega$ and find
  \begin{alignat*}{1}
    \tau_{h}(\nabla \nabla & \mathcal{P}F)_x(X,Y) = [hY,\Phi Y]\Omega(X,JX) -
    [hX,\Phi X] \Omega(JY,Y) \vphantom{\sum_{cycl}}
    \\
    & \ + \Phi X\Big(\sum_{cycl}\Omega([JY,hX],hY)\Big)
    - hX\Big(\sum_{cycl} \Omega([JY,\Phi X],hY) -\Omega([JY,\Phi Y],X)\Big)
    \\
    \notag & \  -\Phi Y\Big(\sum_{cycl}\Omega([JX,hX],hY)\Big)
    +hY\Big(\sum_{cycl}\Omega([JX,\Phi X],Y) -\Omega([JX,\Phi Y],X)\Big),
  \end{alignat*}
  containing only 3\ts{rd} order derivatives of $F$.  The compatibility
  condition is satisfied if and only if
  $\tau_{h}(\nabla \nabla \mathcal{P}F)_x=0$ which is equivalent to equation
  (\ref{third_order_cond}).
\end{proof}
\begin{remark}
  \label{rem:Omega_adapted_P_3}
  In an adapted basis \eqref{eq:basis_1} we have
 \begin{subequations}
  \label{eq:P_3_basis}
   \begin{align}
    \label{eq:S_basis}
     i_\Gamma\Omega&=0 \quad \Longleftrightarrow \quad  \Omega(h_i,h_j)=0, & &  1 \leq i,j \leq n,
     \\
    \label{eq:Gamma_basis}
     i_\Phi\Omega&=0 \quad \Longleftrightarrow  \quad \Omega(v_i,h_j) =0, & & 1 \leq i,j \leq n, \quad i \neq j,
     \\
    \label{eq:Phi_basis}
     i_S\Omega&=0 \quad \Longleftrightarrow \quad  \Omega(v_i,h_n) =
                \Omega(h_i,h_n) =0, & &  1 \leq i \leq n,
   \end{align}
 \end{subequations}
 hence, using the notation $a_{ij}=\Omega (v_i, h_j)$ the only nonzero
 components of $\Omega$ are
  \begin{equation}
    \label{eq:g_ii}
    a_{ii}=\Omega (v_i, h_i), \qquad i = 1, \dots , n\!-\! 1.   
  \end{equation}
  Moreover, because the Hessian of $F$ must be positive quasi-definite, the
  terms in \eqref{eq:g_ii} are positives.
\end{remark}

\begin{corollary}
  \label{adaptedbase_corollary}
  In an adapted basis \eqref{eq:basis_1} the compatibility condition 
  \eqref{third_order_cond} can be expressed as
  \begin{equation}
    \label{eq:phi_loc}
    \beta^i_{ij} (\mathcal{L}_{v_i}a_{ii}) \!+\! \beta^j_{ij}
    (\mathcal{L}_{v_j}a_{jj}) 
    \!+\! \gamma^i_{ij} (\mathcal{L}_{h_i}a_{ii})
    \!+\! \gamma^j_{ij} (\mathcal{L}_{h_j}a_{jj})
    \!+\! \sum_{k=1}^n \! \alpha^k_{ij}a_{kk} =0, 
  \end{equation}
  where $1 \!\leq \! i,j \!\leq \! n$, $a_{ij}\!=\!\Omega(v_i,h_j)$ and the
  summation convention is not applied.  The $\alpha^k_{ij}$, $\beta^k_{ij}$,
  $\gamma^k_{ij}$ are functions in a neighborhood of $x\in \mathcal{T}M$
  determined by the Lie bracket of the elements of the local basis
  \eqref{eq:basis_1}.
\end{corollary}

If $\{\xi^i, \nu^i\}_{1 \leq i\leq n}$ is the dual basis of
\eqref{eq:basis_1}, then $X=\xi^i_Xh_i+\nu^i_Xv_i$ and using $d\Omega=0$ and
from \eqref{eq:P_3_basis} we get
\begin{alignat*}{2}
  \mathcal{L}_{v_i}a_{jj}& =\Omega([v_i,v_j],h_j)\+\Omega([v_j,h_j],v_i)
  \+\Omega([h_j,v_i],v_j)
  && \= \xi^i_{[h_j,v_j]}a_{ii} \+ \big(\nu^j_{[v_i,v_j]}\+
  \xi^j_{[v_i,h_j]}\big)a_{jj},
  \\
  \mathcal{L}_{h_i}a_{jj} & =\Omega([h_i,h_j],v_j)\+\Omega([h_j,v_j],h_i)
  \+\Omega([v_j,h_i],h_j)
  && \=\nu^i_{[h_j,v_j]} a_{ii} \+ \big(\xi^j_{[h_j,h_i]} \+ \nu^j_{[v_j,h_i]}
  \big) a_{jj} ,
\end{alignat*}
where $i,j \in\left\{1,...,n\right\}$ and the summation convention is not
applied. Then \eqref{third_order_cond} can be expressed as
\begin{equation}
  \label{coeff_1}
  \kappa^i_{ij}  a_{ii}+\kappa^j_{ij}a_{jj}+\sum_{k=1}^n\theta^k_{ij}a_{kk}
  +(\lambda_j\!-\!\lambda_i)\big(\mathcal L_{[h_j,v_j]}a_{ii}- 
  \mathcal L_{[h_i,v_i]}a_{jj}\big)=0,
\end{equation}
\noindent
where the functions $\kappa^i_{ij}$ and $\theta^k_{ij}$ are
\begin{subequations}
  \begin{align}
    \label{coeff_2_1}
    \kappa^i_{ij}=& \lambda_j\Big( \mathcal{L}_{v_j}\xi^i_{[h_i,h_j]} \!-\! \,
    \mathcal{L}_{v_j} \nu^i_{[h_j,v_i]} \!-\!  \mathcal{L}_{h_j}
    \nu^i_{[v_i,v_j]} \!+ \!\mathcal{L}_{h_j}\xi^i_{[v_j,h_i]} \Big) +
    \lambda_i\Big( \nu^i_{[v_i,[h_j,v_j]]}
    \\
    & \notag \quad \!-\! \xi^i_{[[h_j,v_j],h_i]}
    \!-\!\nu^j_{[v_j,h_i]}\xi^i_{[v_j,h_j]}
    \!-\!\xi^j_{[h_i,h_j]}\xi^i_{[h_j,v_j]}
    \!-\!\nu^j_{[v_j,v_i]}\nu^i_{[v_j,h_j]}
    \!-\!\xi^j_{[v_i,h_j]}\nu^i_{[h_j,v_j]}\Big)
    \\
    \label{coeff_2_2}
	\theta^k_{ij}=&(\lambda_i-\lambda_j)\Big(\xi^k_{[h_j,v_j]}\nu^k_{[h_i,v_i]}
    -\xi^k_{[h_i,v_i]}\nu^k_{[h_j,v_j]}\Big),
  \end{align}
\end{subequations}
$1 \leq i,j,k \leq n$.  In particular, the coefficients of the third order
terms in \eqref{coeff_1} appearing in the Lie derivative terms are:
\begin{equation}
  \label{eq:3rd_order_coeff}
  \beta_{ij}^i=\lambda_j\nu_{[v_j,h_j]}^i, \qquad \qquad 
  \gamma_{ij}^i=\lambda_j\xi_{[v_j,h_j]}^i.
\end{equation}

\begin{definition}
  We say that the higher order compatibility condition of the spray $S$ is
  reducible, if the coefficients $\beta^i_{ij}$, $\gamma^i_{ij}$ are
  identically zero.
\end{definition}

\begin{remark}
  In the reducible case the third order condition can be identically zero or
  may represent a second order compatibility condition.  From
  \eqref{eq:3rd_order_coeff} it is clear that when the spray has distinct
  Jacobi eigenvalues, then the higher order compatibility condition of the
  spray $S$ is reducible if and only if the eigen-distributions
  $\mathcal{D}_i:=Span\{h_i,v_i\}$, $1 \leq i \leq n $ are involutives.
\end{remark}

\section{The three dimensional case}
\label{sec:4}

In the previous chapter we investigated the integrability condition of the
extended Rapcs\'ak system completed with the curvature condition. We
determined the compatibility condition to lift a second order solution into a
third order solution (Proposition \ref{prop:2_order_integr_cond}) and an extra
higher order compatibility condition appearing to lift a second order solution
into a third order solution (Proposition \ref{prop:3_order_integr_cond}).  In
this chapter we focus on the case, when $M$ is a $3-$dimensional manifold and
the spray is non-isotropic. We consider the generic situation, when the
eigenvalues of the Jacobi endomorphism are pairwise distinct.  We identify the
special cases when the integrability conditions are satisfied and by computing
the higher order Spencer cohomology groups we prove that the system has no
further compatibility condition.

\subsection{Extended Rapcs\'ak system with curvature condition}

Let $S$ be a spray on the 3-manifold $M$ and let us consider the second order
PDE operator $\P$ given by \eqref{eq:P} where the eigenvalues of the Jacobi
tensor $\Phi$ are pairwise different.

\subsubsection*{The first compatibility conditions}

The compatibility condition to lift a second order solution of $\P$ into a
third order solution is given by \eqref{iphiphi} and \eqref{omegavihjhk}.
Then, in the $\dim M=3$ case, we can have the following
\begin{remark}
  \label{thm:comp_1}
  The compatibility condition \eqref{iphiphi} is identically satisfied.
\end{remark}
\noindent
Indeed, $\Phi$ is semi basic 1-1 tensor and from $i_S\Phi =0$ we have
$i_S[\Phi, \Phi]=0$.  Evaluating the semi basic 3-form
$i_{[\Phi, \Phi]}\Omega$ on the 3-dimensional horizontal space by using the
second equation of \eqref{eq:rap_2} we get that
\begin{alignat*}{1}
  i_{[\Phi, \Phi]}\Omega (h_1,h_2,h_3)& = \sum_{123}^{cycl}\Omega\big([\Phi,
  \Phi] (h_1,h_2),h_3\big)= i_S \Omega\big([\Phi, \Phi] (h_1,h_2)\big) =0.
\end{alignat*}

\begin{remark}
  \label{thm:comp_2}
  Introducing $\Phi':=v\!\circ\! [S,\Phi]\!\circ\! h$, the semi basic
  dynamical covariant derivative of $\Phi$ (see \cite{BuMu_2012,
    GrMu_2000}), the integrability condition \eqref{omegavihjhk} can be
  written as
  \begin{equation}
    \label{eq:i_Phi_vesszo}
    i_{\Phi'}\Omega=0.
  \end{equation}
  Indeed, we have
  \begin{math}
    \Phi'(S)=v [S,\Phi]S = v [S,\Phi S] - \Phi ([S,S])=0,
  \end{math}
  thus \eqref{eq:i_Phi_vesszo} is satisfied if and only if
  $i_{\Phi'}\Omega(h_1,h_2)=0$.  Moreover, $F$ being a second order solution,
  we have \eqref{eq:P_3_x}, and in particular $\Omega([v_1,h_2],S)\!=\!0$ and
  $\Omega(h_1,h_2)\!=\!0$ at $x$.  Using $h_i=[J,S]h_i=[v_i,S]-J[h_i,S]$,
  $i=1,2$ and $h_3\!=\!S$ we can obtain that the condition (\ref{omegavihjhk})
  is
  \begin{equation}
    \label{eq:obstr_tau_ijk_dim_3}
    \begin{aligned}
      (\ref{omegavihjhk})
      &=\Omega([v_1,h_2],S)+\Omega([h_2,S],v_1)+\Omega([S,v_1],h_2) 
      \\
      &=\Omega([h_2,S],v_1) -\Omega(J[h_1,S],h_2)
      =\Omega([h_2,S],v_1)-\Omega(v_2,[h_1,S]).
    \end{aligned}
  \end{equation}
  On the other hand, from to $i_\Phi\Omega=0$ we have
  \begin{align}
    \label{of:123}
	\Omega(\Phi[h_i,S],h_j)=\Omega(\Phi h_j,[h_i,S]) 
    =\lambda_j  \Omega(v_j,[h_i,S]), \quad i,j=1,2, \ i\neq j,
  \end{align}
  therefore,
  \begin{alignat*}{1}
    i_{\Phi'}\Omega(h_1,h_2)&=\Omega(\Phi'(h_1),h_2) -\Omega(\Phi'(h_2),h_1)
    \\
    &=\Omega\big([\Phi h_1,S]-\Phi[h_1,S],h_2\big)-\Omega\big([\Phi
    h_2,S]-\Phi[h_2,S],h_1\big) \stackrel{(\ref{of:123})}{=}
    %
    \\ & = \lambda_1\Omega(J[h_1,S],h_2)-\lambda_2\Omega(v_2,[h_1,S])
    -\lambda_2\Omega(J[h_2,S],h_1)+\lambda_1\Omega(v_1,[h_2,S])
    \\
    &=\lambda_1\Omega(v_2,[h_1,S])-\lambda_2\Omega(v_2,[h_1,S])
    -\lambda_2\Omega(v_1,[h_2,S])+\lambda_1\Omega(v_1,[h_2,S])
    \\
    &=(\lambda_1-\lambda_2)(\Omega(v_2,[h_1,S])+\Omega(v_1,[h_2,S])).
  \end{alignat*}
  Comparing the result with (\ref{eq:obstr_tau_ijk_dim_3}) we obtain Remark
  \ref{thm:comp_2}.
\end{remark}


\subsubsection*{Higher order compatibility conditions}

From \cite[Chapter 5]{MiMu_2017} we know that the Spencer cohomology group
$H^{2,2}$ is nonzero and there are extra compatibility conditions arising for
the first prolongation of $\P$. This compatibility condition was calculated in
Proposition \ref{prop:3_order_integr_cond}. Now we compute all the higher
order Spencer cohomology groups $H^{m,2}$ for $m\geq 3$. We have
\begin{lemma}
  \label{t:2ac}
  For any $m\geq 3$ the Spencer cohomology group $H^{m,2}$ is trivial.
\end{lemma}
Let us consider the following Spencer-sequence
\begin{equation}
  \label{spencerseq}
  0 \rightarrow  g_{m+2} \xrightarrow{ \ i \ }T^{*} \otimes 
  g_{m+1} \xrightarrow{ \ \delta^{m}_{1} \ } 
  \Lambda^{2}T^{*} \otimes g_{m} \xrightarrow{ \ \delta^{m}_{2}
    \ }\Lambda^{3}T^{*} \otimes g_{m-1} \longrightarrow \dots
\end{equation}
where $i$ is the inclusion, $\delta^{m}_{1}$ and $\delta^{m}_{2}$ are the
Spencer operators skew-symmetrizing the first two, resp.~three variables.  To
prove the lemma we have to show that
$H^{m,2}=\mathrm{Ker} \,\delta^{m}_{2}/\mathrm{Im} \,\delta^{m}_{1}=0$, that
is $\mathrm{Ker} \,\delta^{m}_{2}=\mathrm{Im} \,\delta^{m}_{1}$.  Since one
has $\mathrm{Im} \,\delta^{m}_{1}\subset \mathrm{Ker} \,\delta^{m}_{2}$, it is
enough to show that the dimension of the two spaces are equal.

\medskip

\textsl{Step 1: computation of $\mathrm{rank} \,\delta^{m}_{1}$.}  Since the
sequence (\ref{spencerseq}) is always exact in the first two term we have
\begin{equation}
  \label{eq:rank_d_1_m}
  \mathrm{rank}\, \delta^{m}_{1}=\dim \, (T^{*} \! \otimes g_{m+1})-\dim \, (g_{m+2}).
\end{equation}
Let us find the general formula for the dimension of $g_{m}$ by determining
how many free tensor components can be chosen to determine its elements.  The
computation is similar to that of $\dim g_{3}$ on page
\pageref{eq:symb_1}. The equations characterizing an element of
$A \in g_{m}(\mathcal{P})$ in $S^{m}T^{*}$ are the prolongations of the
equations \eqref{eq:symb_1}--\eqref{eq:symb_3}:
 \begin{subequations}
  \begin{align}
  \label{eq:PC}
  A(\dots,v_3)&=0,
  \\
  \label{eq:PGamma}
  A(\dots,h_k,v_l)-A(\dots,h_l, v_k)&=0, \quad k,l=1,2,3, \quad k \neq l,
  \\
  \label{eq:PPhi}
    A(\dots,v_k,v_l)&=0, \quad k,l=1,2,3, \quad k \neq l.
  \end{align}
\end{subequations}
Hence the tensor components of the horizontal block
\begin{math}
  \big(A(h_{i_1} \dots  h_{i_m})\big)
\end{math}
are totally symmetric and contains $\Cr{3,m}$ free components. The
\begin{math}
  \big(A_{h_{i_1}}\dots h_{i_{m-1}}, v_{i_m}\big)
\end{math}
block is totally symmetric and if one of the indices is $3$, then that tensor
component is zero. Therefore there are $\Cr{2,m}$ free component in such a
block. In the blocks containing at least two vertical vectors only the
$A(h^k_i,v^{m-k}_i)$, ($i=1,2$, $k=0,\dots, m-2$) are free, where we use the
upper index to denote the multiplicity of the vector. Summing the number of
free components we find
\begin{align*}
  \dim g_{m}=&\Cr{3,m}+\Cr{2,m}+(m-1) \, \Cr{1,m},
\end{align*}
and from \eqref{eq:rank_d_1_m} we get
\begin{equation}
  \mathrm{rank}\, \delta^{m}_{1} =  2n \cdot \dim g_{m+1} - \dim g_{m+2}
  =\frac{5m^{2}\+53m\+38}{2}.
\end{equation}

\medskip

\textsl{Step 2: computation of $\mathrm{Ker} \,\delta^{m}_{2}$.}  We will
compute how many independent tensor components characterize an element $B$ of
$\mathrm{Ker} \, \delta^{m}_{2}$ in $\Lambda^{2}T^{*} \o S^{m}T^{*}$. We can
remark that there is no restriction to the purely horizontal blocks in $g_m(\P)$,
therefore we can use the exact sequence
\begin{displaymath}
  0\longrightarrow S^{m+2}T^{*}_{v}\xrightarrow{ \ i \ }T^{*}_{v} 
  \o S^{m+1}T^{*}_{v}  \xrightarrow{ \ \delta^{m}_{1,v} \ } 
  \Lambda^{2}T^{*}_{v} \otimes S^{m}T^{*}_{v} \xrightarrow{ \ \delta^{m}_{2,v}
    \ } \Lambda^{3}T^{*}_{v} \otimes S^{m-1}T^{*}_{v}\longrightarrow \dots ,
\end{displaymath}
where $\delta^{m}_{1,v}$ and $\delta^{m}_{2,v}$ are the restriction of
$\delta^{m}_{1}$ and $\delta^{m}_2$ on the corresponding subspaces. We get
\begin{align*}
  \nul \delta^{m}_{2,v}=\rank \delta^{m}_{1,v}
  =\dim(T^{*}_{v} \otimes S^{m+1}T^{*}_{v})-\dim(S^{m+2}T^{*}_{v})
  =\C{3,2} \,\Cr{3,m+1}-\Cr{3,m+2},
\end{align*}
and the number $N_0 =\rank \delta^{m}_{2,v}$ of the independent equations
characterizing the horizontal components of $B$ in
$\Lambda^{2}T^{*} \o S^{m}T^{*}$ is
\begin{equation}
  N_0\=\dim \left(\Lambda^{2}T^{*}_{v}
    \o S^{m}T^{*}_{v}\right)\!-\! \nul \delta^{m}_{2,v}
  =\C{3,2}\,\Cr{3,m}\!-\!(\C{3,2}\,\Cr{3,m+1}\!-\!\Cr{3,m+2})=\frac{m^2+m}{2}.
\end{equation}
In the sequel we characterize the number of independent equations on the
tensor components of $B$ having at least one vertical vector in its argument.
For simplicity we also use the notation for the element of the basis
\eqref{eq:basis_1}
\begin{equation}
  \label{eq:basis_2} 
  \left\{e_1,...,e_6\right\}=\left\{h_1,h_2,h_3,v_1,v_2,v_3\right\}.
\end{equation}
In particular, we have $S\=h_3\=e_e$ and $C\=v_3\=e_6$. Depending on which is
more advantageous to express the computation, we use both the notation
\eqref{eq:basis_1} and \eqref{eq:basis_2} for the elements of the basis.  An
element $B\in \Lambda^{2}T^{*}\otimes g_{m}$ is skew-symmetric in the first
two variables, symmetric in the last $m$ variables and from \eqref{eq:PC} --
\eqref{eq:PPhi} we get for any $1 \leq i,j\leq 6$
\begin{subequations}
  \label{eq:1}
  \begin{align}
  \label{eq:e_6}
  B(e_i,e_j,\dots \dots , e_6)&=0,
  \\
  \label{eq:314}
  B(e_i,e_j,\dots,e_1,e_5)-B(e_i,e_j,\dots,e_2,e_4)&=0,
  \\
    \label{eq:343545}
    B(e_i,e_j,\dots, e_3,e_4)=0, \quad
    B(e_i,e_j,\dots,e_3,e_5)=0, \quad
    B(e_i,e_j,\dots,e_4,e_5)&=0.  
  \end{align}
\end{subequations}
In order to simplify even further the computation we use the earlier
introduced upper index notation for the multiplicity in the tensor components
and we set
\begin{displaymath}  
  \E_{ijk}(e_{l_1}^{r_1} \dots  \, e_{l_s}^{r_s})  :=
  \E(e_i,e_j,e_k,e_{l_1}^{r_1} \dots  \, e_{l_s}^{r_s})  
  = \sum^{cycl}_{ijk}B(e_i,e_j,e_k,  e_{l_1}^{r_1} \dots  \, e_{l_s}^{r_s}).
\end{displaymath}
The indexes $1 \leq i,j,k \leq 6$ in $\E_{ijk}$ make always reference to the
indexes with respect to the basis \eqref{eq:basis_2}.  $\E$ is skew-symmetric
in the first three variables and symmetric in the last $m-1$ variables.  Then,
the equations characterizing $\Ker \delta^r_{2}$ are
\begin{equation}
  \label{eq:cal_E}
  \E_{ijk}(e_{l_1}^{r_1} \dots  \, e_{l_s}^{r_s})=0,  \quad r_1+\dots +r_s= m.
\end{equation}
Here we are interested in the equations containing components with at least
one vertical vector. Many of the equations \eqref{eq:cal_E} are trivially
satisfied because the summed terms are zeros, listed in \eqref{eq:1}.  We will
investigate the remaining nontrivial equations whose complete list is
\begin{subequations}
  \label{eq:delta_2}
  \begin{align}
    \label{eq:delta_2_1}
    & \mathcal{E}(e_i,e_j,e_k,v^{m-1}_{1})=0, &
    \\
    \label{eq:delta_2_2}
    & \mathcal{E}(e_i,e_j,e_k,v^{m-1}_{2})=0, &
    \\
    \label{eq:delta_2_3}
    & \mathcal{E}(e_i,e_j,e_k,h^{m-l-1}_{1},v^{l}_{1})=0, & \ l\geq 2, 
    \\
    \label{eq:delta_2_4}
    & \mathcal{E}(e_i,e_j,e_k,h^{m-2}_{1},v_1)=0, &
    \\
    \label{eq:delta_2_5}
    & \mathcal{E}(e_i,e_j,e_k,h^{m-l-1}_{2},v^{l}_{2})=0, & \ l\geq 2, 
    \\
    \label{eq:delta_2_6}
    & \mathcal{E}(e_i,e_j,e_k,h^{m-2}_{2},v_2)=0, &
    \\
    \label{eq:delta_2_7}
    & \mathcal{E}(e_i,e_j,e_k,h^{l}_{1},h^{m-l-2}_{2},v_2)=0, & \ l\geq 1, 
    \\
    \label{eq:delta_2_8}
    & \mathcal{E}(e_i,e_j,e_k,h^{m-2}_{1},v_2)=0, &
    \\
    \label{eq:delta_2_9}
    & \mathcal{E}(h_i,v_j,v_k,h^{l}_{1},h^{s}_{2},h^{m-(l+s+1)}_{3})=0, & 
     s, l \geq 0,
    \\
    \label{eq:delta_2_10}
    & \mathcal{E}(h_i,h_j,v_k,h^{l}_{1},h^{s}_{2},h^{m-(l+s+1)}_{3})=0. &
  \end{align}
\end{subequations}
We will determine how many independent equations are between
\eqref{eq:delta_2}.  Because of the skew-symmetric property of
\eqref{eq:cal_E} in the indices $i,j,k$ it is enough to consider equations
with $i \< j\< k$ in each blocks.  We set
$H^\c\!:=\!\{1, \dots , 6 \}\setminus H$, and whenever $p$ appears in the
formulas below, then it completes the some of the exponents to the value
$m-1$.
\begin{enumerate}
\item 
  Equation \eqref{eq:delta_2_1}: trivially holds if $i,j,k\in \{ 1,4\}^\c$.
  The remaining equations are linearly independent: The pivot tensor
  components are $B(e_i,e_j,e_1, e_4^{m-1})$ for $i<j$, \ $i,j \neq 4$ and
  $B(e_i,e_j,e_4^{m})$ for $i<j$, \ $i,j \in \{ 1,4\}^\c$.  Therefore, we have
  $N_1=16$ in\-de\-pendent equations in this block.

\item 
  Equation \eqref{eq:delta_2_2} is similar to the previous one: It holds
  trivially if $i,j,k\in \{ 2,5\}^\c$.  The remaining equations are linearly
  independent: The pivot tensor components are $B(e_i,e_j,e_5^{m})$ for $i<j$,
  $i,j\neq 5$ and $B(e_i,e_j,e_2,e_5^{m-1})$ for $i<j$, $i,j\in
  \{2,5\}^\c$. Here we have $N_2=16$ independent equations.

\item 
  Equation \eqref{eq:delta_2_3}: For each fixed $l \geq 2$ four equations
  trivially hold when $i,j,k\in \{1,4\}^\c$.  Further six nontrivial relations
  exist:
  \begin{equation}
    \label{of:111}
    \E_{1ij}(e_1^{p},e_4^{l})=\E_{4ij}(e_1^{p+1},e_4^{l-1}), \qquad i\<j,
    \quad i,j\in \{1,4\}^c.  
  \end{equation} 
  For the rest we can observe that the equations with $l$ and $(l-1)$
  $e_4$-exponents are related. The only independent equations are
  $\E_{1jk}(e_1^{m-l},e_4^{l-1})=0$ where $i\<j$, $2\leq l$.  Therefore, we
  have $N_3=\C{5,2}\, \Cr{2,m-4}$ independent equations.
		
\item 
  Equation \eqref{eq:delta_2_5} is similar to the previous one: For each fixed
  $l \geq 2$ there are four trivial equations when $i,j,k \in \{2,5\}^\c$ and
  six nontrivial relations
  \begin{equation}
    \label{of:112}
    \E_{2ij}(e_2^{p},e_5^{l})=\E_{5ij}(e_2^{p+1},e_5^{l-1}), \qquad
    i\<j, \quad i,j\in \{2,5\}^c. 
      \end{equation}
  The equations with $l$ and $(l-1)$ $e_5$-exponents are related and the
  independent equations are $\E_{2jk}(e_2^{m-l},e_5^{l-1})$ where $i\<j$,
  $2\leq l$.  Therefore, we have $N_4=\C{5,2}\, \Cr{2,m-4}$ independent
  equations.
	
\item 
  Equation \eqref{eq:delta_2_4} and \eqref{eq:delta_2_6}: there are $N_5=32$
  independent equations given by $\E_{ijk}(e_1^{m-2},e_4)$ and
  $\E_{ijk}(e_1^{m-2},e_4)$ where $i=1,2$, \ $i\< j \< k$.

\item 
  Equation \eqref{eq:delta_2_7}: when $i,j,k\in \{1,2\}^\c$, then the equation
  is trivially satisfied. Moreover, there are six nontrivial relations in this
  block:
  \begin{equation}
    \label{of:113}
    \E_{1jk}(e_1^{l},e_2^{p},e_5)=\E_{2jk}(e_1^{l+1},e_2^{p-1},e_5),
    \qquad j\<k, \ j,k \in \{1,2\}^c. 
  \end{equation}
  The independent equations are $\E_{1jk}(h_1^{l+1},h_2^{p},v_2)$ where
  $1\<j\<k$, $1\<l$, therefore here there are $N_6=\C{5,2}\cdot \Cr{2,m-4}$
  independent equations.  We remark that the equations with interchanging
  $v_1$ and $v_2$ do not need to be considered because of \eqref{eq:314}.

\item 
  Equation \eqref{eq:delta_2_8}: there are four trivial equations and six more
  nontrivial relations, exactly as in the case of \eqref{eq:delta_2_7}.
  Moreover, there are further relations coming from \eqref{eq:314} since
  \begin{align*}
	\E_{1jk}(e_1^{m-3},e_2,e_4)=\E_{2jk}(e_1^{m-2},e_4), 
    \qquad j<k, \ j,k\in \{1,2\}^c
  \end{align*}
  The independent equations are $\E_{12k}(e_1^{m-2},e_5)$, $3 \leq k \leq 6$,
  so we have here $N_7=4$ independent equations.

\item 
  Equation \eqref{eq:delta_2_9}: For $0\!\leq l \!\leq m\!-\!1$ (resp.
  $0\!\leq \!l,s \!\leq\! m\!-\!1$) we have the following relations:
  \begin{alignat*}{2}
    & \E_{1jk}(h_1^{l},h_2^{p}) =\E_{2jk}(h_1^{l+1},h_2^{p-1})
    +\E_{12k}(e_1^{l},e_2^{p-1},e_j), && j,k \geq 4, 
    \\
    & \E_{1jk}(h_1^{l},h_2^{s},h_3^{p})
    =\E_{3jk}(h_1^{l+1},h_2^{s},h_3^{p-1}), && j,k \geq 4, 
    \\
    & \E_{2jk}(h_1^{l},h_2^{s},h_3^{p})
    =\E_{3jk}(h_1^{l},h_2^{s+1},h_3^{p-1}), && j,k \geq 4, 
    \\
    & \E_{1j6}(h_1^{l},h_2^{p},h_3) =\E_{3j6}(h_1^{l+1},h_2^{p})
    +\E_{136}(h_1^{l},h_2^{p},v_j), && j=4,5, 
    \\
    & \E_{2j6}(h_1^{l},h_2^{p},h_3) =\E_{3j6}(h_1^{l},h_2^{p+1})
    +\E_{136}(h_1^{l-1},h_2^{p+1},v_j), \quad && j=4,5. 
  \end{alignat*}
  The $\E_{1jk}(h_1^{m-1})$, $\E_{3jk}(h_1^l, h_2^s, h_3^p)$, and
  $\E_{2jk}(h_1^l, h_2^p)$, $4 \! \leq \!j \!< \!k \!\leq \!6$,
  $l,s \! \geq \!0$ give $N_8=\C{3,2}+\C{3,2}\Cr{3,m-1}+\C{3,2}\Cr{2,m-1}$
  independent new equations.
\item 
  Equation \eqref{eq:delta_2_10}: for $0 \leq l,s,p\leq m-1$, $l+s+p=m-1$ the
  relations are
  \begin{displaymath}
	\E_{12i}(h^{l}_{1},h^{s}_{2},h^{p}_{3})
    =\E_{13i}(h^{l}_{1},h^{s+1}_{2},h^{p-1}_{3})
	-\E_{23i}(h^{l+1}_{1},h^{s}_{2},h^{p-1}_{3}), \quad   i=4,5,6,
  \end{displaymath}
  where we use that $\E_{123}(h_1^l,h_2^{s},h_3^{p-1},v_i)=0$.  Hence the
  independent equations are $\E_{12i}(h_1^l, h_2^p)$,
  $\E_{13i}(h_1^l, h_2^s, h_3^p)$, $\E_{23i}(h_1^s, h_2^l, h_3^p)$ where
  $0\!\leq \! l$, $1\! \leq \! s$, $l\+s\+p\+1\! \leq \!m$,
  $4 \!\leq \! i \!\leq \! 6$.  Therefore there are
  $N_9=\C{3,1}\, \Cr{2,m-1}+2\, \C{3,1}\, \Cr{3,m-1}$ independent equations.
\end{enumerate}
\medskip

\noindent
Finally, we get that \vspace{-10pt}
\begin{equation}
  \label{eq:nul_d_2}
  \dim \Ker \delta^{m}_{2}=\dim (\Lambda^{2}T^{*}\otimes
  g_{m})-\sum^{9}_{i=0}N_i=\frac{5m^{2}+53m+38}{2}.
\end{equation}
Comparing \eqref{eq:rank_d_1_m} and \eqref{eq:nul_d_2} we get
$\mathrm{rank}(\delta^{m}_{1})= \mathrm{nul} (\delta^{m}_{2})$ therefore,
\begin{math}
  \mathrm{Im} \, \delta^{m}_{1} \!= \!\mathrm{Ker} \, \delta^{m}_{2}
\end{math}
and $H^{m,2}=0$.

\bigskip

\begin{theorem}
  \label{thm:form_int}
  Let $S$ be a non-isotropic spray on a 3-dimensional manifold with distinct
  Jacobi eigenvalues.  Then the PDE operator $\P$ defined in \eqref{eq:P} is
  formally integrable if and only if
  \begin{enumerate}
  \item \label{item:1}
    $\Phi' \in Span \{J, \Phi\}$
  \item \label{item:2} The compatibility condition \eqref{third_order_cond} is
    satisfied.
  \end{enumerate}
\end{theorem}
\begin{proof}
  Using Proposition \ref{prop:2_order_integr_cond} with Remark
  \ref{thm:comp_1} and Remark \ref{thm:comp_2} we get that the first condition
  of Theorem \ref{thm:form_int} guaranties that any $2\ts{nd}$ order solution
  of $\P$ can be prolonged into a $3\ts{rd}$ order solution.  Using
  Proposition \ref{prop:3_order_integr_cond} we get that if the second
  condition of Theorem \ref{thm:form_int} holds then any $3\ts{nd}$ order
  solution of $\P$ can be prolonged into a $4\ts{th}$ order
  solution. Moreover, Lemma \ref{t:2ac} shows that the Spencer cohomology
  groups $H^{m,2}$ are trivial and therefore there is no higher order
  compatibility condition for $\P$ that is, any $m\ts{th}$ order solution can
  be prolonged into a $(m+1)\ts{st}$ order solution, $m\geq 4$.  Therefore, we
  obtain the formal integrability of $\P$.
\end{proof}

\begin{corollary}
  \label{cor:P_met}
  Let $S$ be a non-isotropic analytic spray on a 3-dimensional analytic
  manifold with distinct Jacobi eigenvalues. If the higher order compatibility
  condition is reducible and the conditions \ref{item:1} and \ref{item:2} of
  Theorem \ref{thm:form_int} are satisfied, then $S$ is locally projective
  metrizable.
\end{corollary}

\begin{proof}
  Indeed, in the analytic case, for any initial data there exists a local
  analytic solution of $\P$. Choosing an initial data compatible with the
  positive quasi-definite criteria, one can obtain an analytic solution $F$
  such that $F^2$ is locally positive definite. The spray associated to the
  Finsler function $F$ will be projective equivalent to $S$, that is $S$ is
  locally projective metrizable.
\end{proof}

\bigskip

\subsection{Reducible case: the complete system}
\label{sec:5}

In the previous subsection we investigated the integrability of the extended
Rapcs\'ak system completed with curvature condition. When the conditions
\ref{item:1}.~or \ref{item:2}.~of Theorem \ref{thm:form_int} is not satisfied,
then the PDE system $\P$ corresponding to the projective metrizability is not
integrable.  As usual, the necessary compatibility condition should be added
to the system and restart the analysis of the enlarged system. In this section
we consider the case where this extra compatibility condition is of second
order. 

We remark that the extended Rapcs\'ak system is composed by equations on the
components the 2-form $\Omega=dd_JF$ whose potentially nonzero components are
listed in \eqref{eq:g_ii}. When $\dim M=3$, there are only two such
components: $\Omega(v_1, h_1)$ and $\Omega(v_2, h_2)$.  The extra
compatibility condition gives a new relation between these two components. In
the adapted basis this equation can be written as
\begin{equation}
  \label{eq:reduced_comp}
  \eta_1\, \Omega (v_1,h_1)+\eta_2\,\Omega(v_2,h_2)=0,
\end{equation}
where $\eta_1$ and $\eta_2$ are well defined function on $\TM$ determined by
the spray $S$.
\begin{remark}
  \label{rem:eta_i}
  If \eqref{eq:reduced_comp} is nontrivial but one of the $\eta_i$ vanishes,
  then one of the $a_{ii}=\Omega (v_i,h_i)$, $i=1,2$ must be
  zero. Consequently, $\P$ has no positive quasi-definite solution and the
  spray is not projective metrizable.
\end{remark}

\noindent
In the sequel we suppose that $\eta_1 \neq 0$ and $\eta_2 \neq 0$.  The second
order PDE operator corresponding to \eqref{eq:reduced_comp} will be denoted by
\begin{displaymath}
  P_\Psi: C^\infty(\TM) \to C^\infty(\TM), \qquad 
  P_\Psi(F)= \eta_1\Omega (v_1,h_1)+\eta_2\Omega(v_2,h_2).
\end{displaymath}
The PDE operator corresponding the completed system is $\PP:=(\P,P_\Psi)$
i.e.:
\begin{equation}
  \label{eq:PP}
  \PP=(P_C,  P_{\Gamma}, P_{\Phi},P_\Psi).
\end{equation}

\begin{remark}
  \label{lemma:PP_Omega}
  If $F$ is a solution of $\PP$ then for any $k \!\in\! \mathbb N$ the
  $\mathcal{L}^k\Omega$ (the $k\ts{th}$ order Lie derivatives of $\Omega$) can
  be calculated \emph{algebraically} from $\Omega$.
\end{remark}
\begin{proof}
  It is sufficient to check this property for the Lie derivatives with respect
  to the element of the adapted basis \eqref{eq:basis_1}. According the Remark
  \ref{rem:Omega_adapted_P_3}, the only nonzero components of $\Omega$ are
  $a_{11}=\Omega(v_1, h_1)$ and $a_{22}=\Omega(v_2, h_2)$.  Using $d\Omega=0$
  and equation \eqref{eq:P_3_basis} with \eqref{eq:reduced_comp} we can find
  for the Lie derivatives of $a_{11}$ the following:
  \begin{alignat*}{1}
    \mathcal{L}_{h_2}a_{11}\=h_2\Omega(v_1,h_1) &\= \Omega([h_2,v_1],h_1) \+
    \Omega([v_1,h_1],h_2) \+ \Omega([h_1, h_2],v_2),
    \\
    \mathcal{L}_{v_2}a_{11}\=v_2\Omega(v_1,h_1) &\= \Omega([v_2,v_1],h_1) \+
    \Omega([v_1,h_1],v_2) \+ \Omega([h_1,v_2],v_1),
    \\
    \mathcal{L}_{h_1}a_{11}\=h_1\Omega(v_1,h_1) &
    \=h_1\big(\tfrac{\eta_2}{\eta_1}\Omega(v_2,h_2)\big) \=
    h_1\big(\tfrac{\eta_2}{\eta_1}\big) \Omega(v_2,h_2) \+
    \tfrac{\eta_2}{\eta_1} h_1 \Omega(v_2,h_2)
    \\
    & \= h_1\big(\tfrac{\eta_2}{\eta_1}\big) \Omega(v_2,h_2) \+
    \tfrac{\eta_2}{\eta_1} \big(\Omega([h_1,v_2],h_2) \+ \Omega([v_2,h_2],h_1)
    \+ \Omega([h_2,h_1], v_2) \big),
    \\
    \mathcal{L}_{v_1}a_{11}\= v_1\Omega(v_1,h_1) &
    \=v_1\big(\tfrac{\eta_2}{\eta_1}\Omega(v_2,h_2)\big) \=
    v_1\big(\tfrac{\eta_2}{\eta_1}\big) \Omega(v_2,h_2) \+
    \tfrac{\eta_2}{\eta_1} v_1 \Omega(v_2,h_2)
     \\
     & \= v_1\big(\tfrac{\eta_2}{\eta_1}\big) \Omega(v_2,h_2) \+
     \tfrac{\eta_2}{\eta_1} \big(\Omega([v_1,v_2],h_2) \+
     \Omega([v_2,h_2],v_1) \+ \Omega([h_2,v_1], v_2) \big).
  \end{alignat*}
  On the right hand side of the above expressions there are only terms
  containing $\Omega$ but not its derivatives.  We remark that all these terms
  can be expressed as linear combinations of $a_{11}$ and $a_{22}$.  Any
  further derivatives can be computed the same way. We can find the formula
  for the derivatives of $a_{22}$ by interchanging the indexes
  $1 \leftrightarrow 2$.
\end{proof}

\begin{corollary}
  The system \eqref{eq:PP} is complete in the sense that either all
  compatibility conditions are satisfied or the spray is not projective
  metrizable.
\end{corollary}
\begin{proof}
  Indeed, according to Remark \ref{lemma:PP_Omega}, if there is a non-trivial
  extra compatibility condition for $\PP$ then it can be expressed
  algebraically with $\Omega$, that would give a new and independent linear
  (homogeneous) equation between $a_{11}$ and $a_{22}$. Consequently, from
  Remark \ref{rem:Omega_adapted_P_3} we get that only the trivial solution
  ($a_{ij}=0$) exists.
\end{proof}

\subsection*{Compatibility conditions}

To compute the compatibility conditions of $\PP$, we can follow the methods
presented in the previous chapter.  The symbol of $\PP$ and its prolongations
are
\begin{alignat*}{2}
  \sigma_{2}(P_\Psi): & S^2T^* \longrightarrow \mathbb{R}, \quad &&
  \big(\sigma_2(P_\Psi)A^2\big) = f_1A^2(v_1,v_1)+f_2A^2(v_2,v_2),
  \\
  \sigma_{3}(P_\Psi): & S^3T^* \longrightarrow T^{*}, \quad &&
  \big(\sigma_3(P_\Psi)A^3\big)(X) = f_1A^3(X,v_1,v_1)+f_2A^3(X,v_2,v_2),
  \\
  \sigma_{4}(P_\Psi): & S^4T^* \longrightarrow S^2T^{*}, \quad &&
  \big(\sigma_4(P_\Psi)A^4\big)(X) =
  f_1A^4(X,Y,v_1,v_1)+f_2A^4(X,Y,v_2,v_2),
\end{alignat*}
$A^k\in S^{k}T^{*}$ and $X,Y \in T$. Using the map $\tau$ and $\tau_h$ defined
in \eqref{eq:tau} and \eqref{highero} respectively, we can consider the
extended obstruction map:
\begin{displaymath}
  \widetilde{\tau}:=(\tau, \, \widetilde{\tau}_{1}, \, \widetilde{\tau}_{2}), 
  \qquad   
  \widetilde{\tau}^1:=(id \o \widetilde{\tau}, \, \tau_h, \, 
  \widetilde{\tau}_{3}, \,  \widetilde{\tau}_{4}, \, 
  \widetilde{\tau}_{5}, \, \widetilde{\tau}_{6})
\end{displaymath}
where
\begin{alignat*}{2}
  &\widetilde{\tau}_{1} (B)=B_\Psi(C) \! - \!  f_1B_{C}(v_1,v_1) \! - \!
  f_2B_{C}(v_2,v_2),
  \\
  &\widetilde{\tau}_{2} (B)=B_\Psi(S) \! - \! \tfrac{f_1}{2}B_{\Gamma}(v_1, S,h_1)
  \! - \! \tfrac{f_2}{2}B_{\Gamma}(v_2, S,h_2) \! - \!
  \tfrac{f_1}{\lambda_1}B_{\Phi}(h_1,h_1,S)\! - \!
  \tfrac{f_2}{\lambda_2}B_{\Phi}(h_2,h_2,S),
  \\
  &\widetilde{\tau}_{3} (\widetilde{B})=\widetilde{B}_\Psi(v_1,v_2) \! - \!
  \tfrac{f_1}{\lambda_1\! - \! \lambda_2}\widetilde{B}_{\Phi}(v_1,v_1,h_1,h_2)
  \! - \! \tfrac{f_2}{\lambda_1\! - \!
    \lambda_2}\widetilde{B}_{\Phi}(v_2,v_2,h_1,h_2), \tfrac{}{}
  \\
  &\widetilde{\tau}_{4} (\widetilde{B})=\widetilde{B}_\Psi(h_1,h_2) \! - \!
  \tfrac{f_1}{2}\widetilde{B}_{\Gamma}(h_1,v_1,h_2,h_1) \! - \!
  \tfrac{f_2}{2}\widetilde{B}_{\Gamma}(h_2,v_2,h_1,h_2)
  \\
  & \qquad \qquad \qquad\qquad \quad\! - \! \tfrac{f_1}{\lambda_1\! - \!
    \lambda_2}\widetilde{B}_{\Phi}(h_1,h_1,h_1,h_2) \! - \!
  \tfrac{f_2}{\lambda_1\! - \!
    \lambda_2}\widetilde{B}_{\Phi}(h_2,h_2,h_1,h_2),
  \\
  &\widetilde{\tau}_{5} (\widetilde{B})=\widetilde{B}_{\Psi}(h_1,v_2) \!  - \!
  \tfrac{f_1}{\lambda_1\! - \!
    \lambda_2}\widetilde{B}_{\Phi}(h_1,v_1,h_1,h_2)
  \! - \! \tfrac{f_2}{2}\widetilde{B}_{\Gamma}(v_2,v_2,h_1,h_2) \! - \!
  \tfrac{f_2}{\lambda_1\! - \! \lambda_2}\widetilde{B}_{\Phi}(v_2,h_2,h_1,h_2),
  \\
  & \widetilde{\tau}_{6} (\widetilde{B})=\widetilde{B}_{\Psi}(v_1,h_2)
  +\tfrac{f_1}{2}\widetilde{B}_{\Gamma}(v_1,v_1,h_1,h_2)
  \! - \! \tfrac{f_1}{\lambda_1\! - \!  \lambda_2}
  \widetilde{B}_{\Phi}(h_1,v_1,h_1,h_2) \! - \!  \tfrac{f_2}{\lambda_1\! - \!
    \lambda_2}\widetilde{B}_{\Phi}(h_2,v_2,h_1,h_2),
\end{alignat*} 
and
\begin{math} 
  B\!=\!(B_C,B_{\Gamma},B_{\Phi},B_\Psi) 
\end{math} 
denotes an element of
\begin{math} 
  T^{*} \o \big(T^* \!\times\! \Lambda^{2}T^{*}_{v} \!\times\!
  \Lambda^{2}T^{*}_{v} \!\times \!T^{*}\big)
\end{math} 
and
\begin{math} 
  \widetilde{B}= (\widetilde{B}_C,\widetilde{B}_{\Gamma},\widetilde{B}_{\Phi},
  \widetilde{B}_\Psi)
\end{math} 
denotes an element of
\begin{math} 
  S^2T^{*} \o \big(T^* \!\times\! \Lambda^{2}T^{*}_{v} \!\times\!
  \Lambda^{2}T^{*}_{v} \!\times \!T^{*}\big).
\end{math} 
Then
\begin{equation}
  \label{egzakt22} 
  \mathrm{Im}\, \sigma_3(\PP) = \mathrm{Ker} \,  \widetilde{\tau},
  \qquad \mathrm{Im}\, \sigma_4(\PP) = \mathrm{Ker} \,   \widetilde{\tau}^1.
\end{equation} 
The compatibility conditions of the PDE operator $\PP$ can be calculated:
\begin{alignat*}{1} 
  \widetilde{\tau}_{1} (\nabla \PP F) &=(\mathcal{L}_Cf_1)\,\Omega(v_1,h_1)
  +(\mathcal{L}_Cf_2)\, \Omega(v_2,h_2),
  \\
  \widetilde{\tau}_{2} (\nabla \PP F)
  &=2f_1\Omega([S,v_1],h_1)+2f_2\Omega([S,v_2],h_2) +
  (\mathcal{L}_Sf_1)\Omega(v_1,h_1) + (\mathcal{L}_Sf_2)\Omega(v_2,h_2),
  \\
  \widetilde{\tau}_{3} (\nabla^2 \PP F) & =(v_1(v_2f_1))\Omega(v_1,h_1)
  +(v_1(v_2f_2))\Omega(v_2,h_2) +f_2[v_1,v_2]\Omega(v_2,h_2)
  \\
  & +(v_2f_1)v_1\Omega(v_1,h_1) +(v_1f_1)\big(\csum\Omega([v_2,v_1],h_1)\big)
  +(v_2f_2)\big(\csum\Omega([v_1,v_2],h_2)\big)
  \\
  & +(v_1f_2)v_2\Omega(v_2,h_2) +f_1v_1\big(\csum\Omega([v_2,v_1],h_1)\big)
  -f_2v_2\big(\csum\Omega([v_2,v_1],h_2)\big),
  \\
  \widetilde{\tau}_{4} (\nabla^2 \PP F) &= (h_1(h_2f_1))\Omega(v_1,h_1) +
  (h_2f_1)h_1\Omega(v_1,h_1) + (h_1f_1)h_2\Omega(v_1,h_1)
  \\
  &+ h_1(h_2f_2)\Omega(v_2,h_2) + (h_2f_2)h_1\Omega(v_2,h_2) +
  (h_1f_2)h_2\Omega(v_2,h_2)
  \\
  & + f_2[h_1,h_2]\Omega(v_2,h_2) + f_1h_1\big(\csum\Omega([h_2,v_1],h_1)\big)
  -f_2h_2\big(\csum\Omega([v_2,h_1],h_2)\big),
  \\
  \widetilde{\tau}_{5} (\nabla^2 \PP F) &= (h_1(v_2f_1))\Omega(v_1,h_1) +
  (v_2f_1)h_1\Omega(v_1,h_1) + (h_1f_1)v_2\Omega(v_1,h_1)
  \\
  & + (h_1(v_2f_2)) \Omega(v_2,h_2) + (v_2f_2)h_1\Omega(v_2,h_2) +
  f_1h_1\big(\csum\Omega([v_2,v_1],h_1)\big)
  \\
  & + (h_1f_1)\Omega(v_2,h_2)+ f_2[h_1,v_2]\Omega(v_2,h_2)
  -f_2v_2\big(\csum\Omega([v_2,h_1],h_2)\big),
  \\
  \widetilde{\tau}_{6} (\nabla^2 \PP F) &=(v_1(h_2f_1))\Omega(v_1,h_1) +
  (h_2f_1)v_1\Omega(v_1,h_1) + (v_1f_1)h_2\Omega(v_1,h_1)
  \\
  & + (v_1(h_2f_2))\Omega(v_2,h_2) + (h_2f_2)v_1\Omega(v_2,h_2) +
  f_1v_1\big(\csum\Omega([h_2,v_1],h_1)\big)
  \\
  & + (v_1f_2)h_2\Omega(v_2,h_2) +
  f_2[v_1,h_2]\Omega(v_2,h_2)-f_2h_2\big(\csum\Omega([v_2,v_1],h_2)\big).
\end{alignat*}
Then, using Remark \ref{lemma:PP_Omega}, the above expressions can be written
as
\begin{equation}
  \label{eq:comp_PP}
  \begin{aligned}
    \widetilde{\tau}_{i} (\nabla \PP F)=\eta_1^i \Omega(v_1, h_1) +\eta_2^i
    \Omega(v_2, h_2) & & i&=1,2,
    \\
    \widetilde{\tau}_{j} (\nabla^2 \PP F)=\eta_1^j \Omega(v_1, h_1) +\eta_2^j
    \Omega(v_2, h_2) & & j&=3,4,5,6
  \end{aligned}
\end{equation}
where $\eta_i^k$ are functions on $\TM$.  If we consider the matrix
\begin{displaymath}
  \Theta = \left(
    \begin{matrix}
      \eta_1 & \eta_1^1 & \dots & \eta_1^6
      \\
      \eta_2 & \eta_2^1 & \dots & \eta_2^6
  \end{matrix}
  \right),
\end{displaymath}
we can have the following
\begin{proposition}
  \label{prop:PP_int}
  The operator $\PP$ is formally integrable if and only if
  $\mathrm{rank}\, \Theta=1$.
\end{proposition}
\begin{proof}
  Applying the method used in the previous chapter we get that a 2\ts{nd}
  order solution $F$ can be lifted into a third order solution iff
  $\tau_4(\nabla \PP F)=0$ and a 3\ts{rd} order solution can be lifted into a
  third order solution iff $\tau^1_4(\nabla^2 \PP F)=0$.  The compatibility
  conditions expressed in terms of $\Omega$ can be written as a system of
  (homogeneous) linear algebraic system equating the right hand side of
  \eqref{eq:comp_PP} to zero. These equations are identically satisfied if and
  only if they are multiple of equation \eqref{eq:reduced_comp}, that is the
  $\mathrm{rank}\, \Theta=1$. Moreover, it is easy to show that
  \eqref{eq:basis_1} is a quasi-regular basis for the first prolongation of
  the symbol of $\PP$, that is the Cartan's test is satisfied.  From the
  Cartan-K{\"a}hler theorem (Theorem \ref{Cartan_Kahler}) we get that $\PP$ is
  formally integrable.
\end{proof}
Based on Remark \ref{rem:eta_i} and Proposition \ref{prop:PP_int} we have
the following
\begin{theorem}
  Let $S$ be a non-isotropic analytic spray on a 3-dimensional analytic
  manifold with distinct Jacobi eigenvalues. If
  $\Phi' \not \in Span \{J, \Phi\}$ or the compatibility condition
  \eqref{third_order_cond} is reducible but not identically zero, then the
  spray is locally projective metrizable if and only if $\eta_1\cdot \eta_2<0$
  and $\mathrm{rank}\, \Theta=1$.
\end{theorem}
\begin{proof}
  Under the hypothesis of the theorem, the operator $\PP$ corresponding to the
  projective metrizability is integrable. Considering at any point $x\in \TM$
  a second order solution $j_{2,x}$ having $a_{11}> 0$ and $a_{22}> 0$ one can
  extend it into an analytic positive quasi-definite solution. Therefore, the
  spray is locally projective metrizable.
\end{proof}

\subsection*{Example} 

Let us consider a spray locally described by the system of second order
differential equation
\begin{equation}
  \label{eq:example}
  \ddot{x}_1=f^1\Big(x_1,x_3,\frac{\dot{x}_1}{\dot{x}_3}\Big)\!\cdot \!
  {\dot{x}_3}^{2}, \qquad
  \ddot{x}_2=f^2\Big(x_2,x_3,\frac{\dot{x}_2}{\dot{x}_3}\Big) \!\cdot
  \!{\dot{x}_3}^{2}, \qquad 
  \ddot{x}_3=f^3(x_3) \!\cdot \! \dot{x}_3^{2},
\end{equation}
where $f^1$, $f^2$ and $f^3$ are analytic functions. The non-zero connection
coefficients
\begin{math}
  \Gamma^i_j=-\frac{1}{2}\frac{\partial f^i}{\partial y^j}
\end{math}
are 
\begin{alignat*}{3}
  \Gamma^{1}_{1}&=-\frac{1}{2} \partial_3 f^1 \! \cdot \! y_3, \qquad &
  \Gamma^{2}_{2}&=-\frac{1}{2}\partial_3f^2 \!\cdot \! y_3, \qquad &
  \\
  \Gamma^{1}_{3}&=\hphantom{-}\frac{1}{2}\partial_3f^1 \!\cdot \! y_1-f^1
  \!\cdot \!  y_3, \qquad &
  \Gamma^{2}_{3}&=\hphantom{-}\frac{1}{2}\partial_3f^2 \!\cdot \!  y_2 - f^2
  \!\cdot \!  y_3, \qquad & \Gamma^{3}_{3}&=-f^3 \!\cdot \! y_3.
\end{alignat*}
Then components of the Jacobi endomorphism can be computed by the formula
\begin{math}
  \Phi^i_j=-\frac{\partial f^i}{\partial x^j} - \Gamma^i_l\Gamma^l_j -
  S(\Gamma^i_j). 
\end{math}
The matrix
\begin{math}
  (\Phi^i_j)=\left(
    \begin{smallmatrix}
      \lambda_1 & 0 & \star \\
      0 & \lambda_2  & \star \\
      0 & 0 & 0
    \end{smallmatrix}
  \right)
\end{math}
is upper triangular where the eigenvalues are
\begin{alignat*}{1}
  \lambda_i = \ & \frac{1}{2} \!\left( \partial_{13} f^i \!-\!f^3 \partial_{33} f^i
  \right)y_1y_3
  + \frac{1}{2}\left(\!\partial_{23} f^i \!+\!f^i \partial_{33}f^i
    \!+\!f^3\partial_{3}f^i \! -\! 2\partial_1 f^i \! -\!
    \frac{1}{2}(\partial_3 f^i)^2 \right)y_3^2, 
\end{alignat*}
for $i=1,2$ and $\lambda_3\!=\!0$. The semibasic eigenvectors are
\begin{alignat*}{2}	
  h_1&=\frac{\partial}{\partial x_1}+\frac{1}{2}y_3\partial_3f^1
  \frac{\partial}{\partial y_1}, \qquad
  & v_1&=\frac{\partial}{\partial y_1},
  \\
  h_2&=\frac{\partial}{\partial x_2}+\frac{1}{2}y_3\partial_3f^2
  \frac{\partial}{\partial y_2}, \qquad
  & v_2&=\frac{\partial}{\partial y_2},
  \\
  h_3&= \sum_{i=1}^3 y_i\frac{\partial}{\partial
    x_i}+\sum_{i=1}^3y_3^2 f^i \frac{\partial}{\partial y_i}, \qquad 
  & v_3&=\sum_{i=1}^3y^i\frac{\partial}{\partial y_i},
\end{alignat*}
where $v_3\!=\!C$ and $h_3\!=\!S$ is the spray corresponding to
\eqref{eq:example}.

When $\lambda_1,\lambda_2$ are zero or $\lambda_1,\lambda_2$ are nonzero with
$\lambda_1=\lambda_2$, then the spray is flat or isotropic, therefore it is
projective metrizable (see \cite[Corollary 4.7]{MiMu_2017}).  When
$\lambda_1,\lambda_2$ are nonzero with $\lambda_1\neq \lambda_2$, then the
spray $S$ is not isotropic.  Computing the semi basic dynamical covariant
derivative of $\Phi$ we get
\begin{math}
  \Phi'(h_i)=v[h_i,S]=\mu_i v_i,
\end{math}
($i=1,2,3$) where
\begin{alignat*}{1}
  \mu_i= \frac{1}{4} \left(y^3y^3 \big((\partial_3f^i)^2 +4\partial_1f^i
    -2f^i\partial_{33}f^i -2f^3\partial_3f^i\big) +2y^iy^3f^3\partial_{33}f^i
  \right), \qquad i =1,2,
\end{alignat*}
and $\mu_3=0$. Therefore $\Phi'= A \Phi + B J$ with
$A\!=\!(\mu_1\!-\!\mu_2)/(\lambda_1\!-\!\lambda_2)$ and
$B\!=\!(\lambda_1\mu_2\!-\!\lambda_2\mu_1)/(\lambda_1\!-\!\lambda_2)$, that is
$\Phi'\in Span\{\Phi, J\}$.  Moreover, one has 
\begin{displaymath}
  [v_1,h_2]=0, \quad [v_2,h_1]=0, \quad [h_1,h_2]=0, \quad [v_1,h_1] 
  \in \mathrm{Span}(v_1), \quad [v_2,h_2]  \in \mathrm{Span}(v_2), 
\end{displaymath}
therefore the functions $\kappa^i_{ij}$, $\theta^i_{ij}$, $\beta_{ij}^i$,
$\gamma_{ij}^i$ in \eqref{coeff_2_1}, \eqref{coeff_2_2} and
\eqref{eq:3rd_order_coeff} vanish. Hence we have \eqref{eq:phi_loc} and from
Corollary \ref{adaptedbase_corollary} we can obtain that the compatibility
condition \eqref{third_order_cond} is satisfied.  Using Theorem
\ref{thm:form_int} we get that $\P$ defined in \eqref{eq:P} is formally
integrable and from Corollary \ref{cor:P_met} we obtain that
\eqref{eq:example} is locally projective metrizable.

\bigskip

\noindent
Authors’ addresses:

\bigskip

\noindent
Tam\'as Milkovszki,
\\
Institute of Mathematics, University of Debrecen,
\\
H-4032 Debrecen, Egyetem t\'er 1, Hungary
\\
\emph{E-mail address:} {\tt milkovszki@science.unideb.hu}

\medskip

\noindent
Zolt\'an Muzsnay,
\\
Institute of Mathematics, University of Debrecen,
\\
H-4032 Debrecen, Egyetem t\'er 1, Hungary
\\
\emph{E-mail address:} {\tt muzsnay@science.unideb.hu}
\\
\emph{URL}: {\tt math.unideb.hu/muzsnay-zoltan}

\end{document}